\documentclass[a4paper,11pt]{amsart}

\usepackage{amsfonts, latexsym, amsmath, amssymb, amsthm, amscd}
\usepackage[all]{xy}
\usepackage[english]{babel}
\usepackage[utf8]{inputenc}
\usepackage[T1]{fontenc}
\usepackage{graphicx}
\usepackage{booktabs}
\usepackage{array, tabularx}
\usepackage{setspace}
\usepackage{textcomp}
\usepackage{tikz}
\usepackage{multirow}
\usepackage{paralist}
\usepackage{comment}
\usepackage{url}
\usepackage{tensor}

\usepackage[hypertexnames=false,
backref=page,
    pdftex,
    pdfpagemode=UseNone,
    breaklinks=true,
    extension=pdf,
    colorlinks=true,
    linkcolor=blue,
    citecolor=blue,
    urlcolor=blue,
]{hyperref}

\usepackage[top=1.5in, bottom=.9in, left=0.9in, right=0.9in]{geometry}

\newcommand{\referenza}{}

\newtheorem{thm}{Theorem}[section]
\newtheorem*{thm*}{Theorem \referenza}

\newtheorem*{cor*}{Corollary \referenza}

\newtheorem*{lem*}{Lemma \referenza}

\newtheorem{prop}[thm]{Proposition}
\newtheorem*{prop*}{Proposition \referenza}

\newtheorem*{conj*}{Conjecture \referenza}

\newtheorem{defi}[thm]{Definition}

\DeclareMathOperator{\im}{im}

\newcommand{\R}{\mathbb{R}}
\newcommand{\C}{\mathbb{C}}
\newcommand{\del}{\partial}
\newcommand{\delbar}{\overline{\partial}}

\numberwithin{equation}{section}

\newcommand{\etalchar}[1]{$^{#1}$}

\allowdisplaybreaks

\title[Geometric formalities along CRF]{Geometric formalities along the Chern-Ricci flow}

\author{Daniele Angella}
\address[Daniele Angella]{
Dipartimento di Matematica e Informatica ``Ulisse Dini''\\
Universit\`a degli Studi di Firenze\\
viale Morgagni 67/a\\
50134 Firenze, Italy
}
\email{daniele.angella@gmail.com}
\email{daniele.angella@unifi.it}

\author{Tommaso Sferruzza}
\address[Tommaso Sferruzza]{
Dipartimento di Matematica e Informatica ``Ulisse Dini''\\
Universit\`a degli Studi di Firenze\\
viale Morgagni 67/a\\
50134 Firenze, Italy
}
\email{tommaso.sferruzza@stud.unifi.it}

\thanks{The first author is supported by the Project PRIN ``Varietà reali e complesse: geometria, topologia e analisi armonica'', by SIR2014 project RBSI14DYEB ``Analytic aspects in complex and hypercomplex geometry'', and by GNSAGA of INdAM}

\date{\today}

\begin{document}

\begin{abstract}
In this paper, we study how the notions of geometric formality according to Kotschick and other geometric formalities adapted to the Hermitian setting evolve under the action of the Chern-Ricci flow on class VII surfaces, including Hopf and Inoue surfaces, and on Kodaira surfaces.
\end{abstract}

\maketitle

\section*{Introduction}
This note is related to the problem of understanding  the algebraic structure of cohomologies of complex manifolds.

On a differentiable manifold $X$, the differential complex of forms has a structure of algebra. The wedge product induces an algebra structure in de Rham cohomology, but {\em not in a uniform way}. We mean that, in general, it is not possible to choose a system of representatives having itself a structure of algebra.
This is evident if we want to choose harmonic representatives with respect to a Riemannian metric on a compact differentiable manifold: indeed, in the words of Sullivan, there is an “incompatibility of wedge products and harmonicity of forms” \cite[page 326]{sullivan}. In general, on a compact differentiable manifold $X$, the choice of representatives for the de Rham cohomology yields just a structure of $A_{\infty}$-algebra in the sense of Stasheff \cite{stasheff}, by the Homotopy Transfer Principle by Kadeishvili \cite{kadeishvili}, see {\itshape e.g.} \cite{zhou, merkulov}. We refer to \cite{lu-palmieri-wu-zhang, bujis-morenofernandez-murillo} for understanding the relationship between the higher multiplications and the Massey products. When such an $A_\infty$-algebra is actually an algebra, we say that $X$ is {\em formal} in the sense of Sullivan \cite{sullivan}. In this case, the differential graded algebra of forms and the dga of de Rham cohomology share the same minimal model, which contains information on the rational homotopy: therefore the rational homotopy type of $X$ ``can be computed formally from'' the cohomology ring $H^\bullet_{dR}(X;\mathbb R)$; see \cite{sullivan} for details. As a special case, when we can fix a Riemannian metric $g$ such that the {\em harmonic} representatives with respect to $g$ have a structure of algebra, say that $X$ is {\em geometrically formal} in the sense of Kotschick \cite{kotschick}.

On a complex manifold $X$, the double complex of forms $(\wedge^{\bullet,\bullet}X, \partial, \overline\partial)$ has a structure of bi-differential bi-graded algebra. Then it is possible to investigate analogous notions for the Dolbeault cohomology. Neisendorfer and Taylor introduced the notion of {\em (strictly) Dolbeault formality} and ``complex homotopy groups'' in \cite{neisendorfer-taylor}; for Hermitian manifolds, the notion of {\em (strictly) geometric-Dolbeault formality} has been investigated by Tomassini and Torelli in \cite{tomassini-torelli}.

Besides Dolbeault cohomology, Bott-Chern $H^{\bullet,\bullet}_{BC}(X):=\frac{\ker\partial\cap\ker\overline\partial}{\mathrm{im}\,\partial\overline\partial}$ \cite{bott-chern} and Aeppli $H^{\bullet,\bullet}_{A}(X):=\frac{\ker\partial\overline\partial}{\mathrm{im}\,\partial+\mathrm{im}\,\overline\partial}$ \cite{aeppli} cohomologies provide further cohomological invariants on complex manifolds. They are directly related by morphisms to both the Dolbeault and the de Rham cohomology, and allow to numerically characterize the strong Hodge decomposition given by the $\partial\overline\partial$-Lemma \cite{angella-tomassini-3, angella-tardini}.
The notion of  {\em geometrically-Bott-Chern formality} for Hermitian manifolds is introduced and studied in \cite{angella-tomassini-Formality}.
Here, {\em Bott-Chern-harmonic} forms are in the kernel of the fourth-order Bott-Chern Laplacian introduced in \cite{kodaira-spencer-3, schweitzer}.
On the other side, it is not clear how to study a notion of ``Bott-Chern and Aeppli formality''. However, the Bott-Chern cohomology has a structure of algebra, while the Aeppli cohomology is just a $H_{BC}(X)$-module. We refer to the discussion in \cite{angella-tardini, angella-Bielefeld}.

In this note, we focus on geometric formalities of complex manifolds and its dependence on the Hermitian metric. In \cite{tomassini-torelli, tardini-tomassini}, the authors study the behaviour of Dolbeault formality, respectively geometric-Bott-Chern formality, under small deformations of the complex structure.
Here, we keep the complex structure fixed, and we study geometric formalities with respect to Hermitian metrics evolving along a geometric flow. More precisely, we consider the {\em Chern-Ricci flow} \cite{gill, tosatti-weinkove-JDG} that evolves Hermitian structures $\omega(t)$ by the Chern-Ricci form,
$$ \frac{\partial}{\partial t} \omega = - \mathrm{Ric}^{Ch}(\omega) , $$
and we study the possible algebra structure on the space of (de Rham, Dolbeault, Bott-Chern, Aeppli) harmonic forms with respect to $\omega(t)$ varying $t$.

We study in details geometric formality according to Kotschick for a whole class of surfaces evolving by the Chern-Ricci flow, \emph{i.e.} compact complex non-K\"ahler surfaces with Kodaira dimension $\text{Kod}(X)=-\infty$ and first Betti number $b_1(X)=1$, known as \emph{class VII} of the Enriques-Kodaira classification. In particular, we first rule out class VII surfaces with second Betti number $b_2>0$ by applying arguments as in \cite{kotschick}. Then, we exploit the structure of quotients of Lie groups with invariant complex and Hermitian structure on the only class VII surfaces with $b_2=0$, that is Hopf and Inoue surfaces see \cite{bogomolov, kodaira-1, li-yau-zheng, teleman-ijm}, in order to reduce the description of harmonic forms and the equation of the Chern-Ricci flow of such surfaces at the level of invariant forms and thus make explicit computations. We obtain the following.

\renewcommand{\referenza}{\ref{thm:main}}
\begin{thm*}
On class VII surfaces, the Chern-Ricci flow preserves the notion of geometric formality according to Kotschick starting at initial invariant metrics.
\end{thm*}

We also study the evolution of geometric formality according to Kotschick on other compact complex non-K\"ahler surfaces that are diffeomorphic to solvmanifolds, \emph{e.g.} Kodaira surfaces. Since any complex structures on such surfaces is left-invariant, see \cite[Theorem 1]{hasegawa-jsg}, we focus on invariant forms also in this case.

\renewcommand{\referenza}{\ref{prop:Kodaira-1}}
\begin{prop*}
On Kodaira surfaces, geometric formality according to Kotschick is preserved by the Chern-Ricci flow starting at initial invariant metrics.
\end{prop*}

We note that, also in this case, it is possible to rule out primary Kodaira surfaces by the obstructions in \cite{kotschick} or \cite{hasegawa-pams}, and therefore we focus on secondary Kodaira surfaces with initial invariant metrics.

Regarding Dolbeault and Bott-Chern geometric formalities evolving by the Chern-Ricci flow, by applying the analogous procedure on Hopf, Inoue, and Kodaira surfaces, we have reached results as follows. We also checked how the algebraic structures of Aeppli cohomology and its harmonic representatives are modified along the Chern-Ricci flow.

\renewcommand{\referenza}{\ref{prop:db-bc}}
\begin{prop*}
On Hopf, Inoue, and Kodaira surfaces, the notions of Dolbeault geometric formality and Bott-Chern geometric formality are all preserved by the Chern-Ricci flow starting at initial invariant metrics. Moreover the properties of Aeppli-harmonic forms having a structure of algebra or a structure of $H_{BC}$-module are all preserved by the Chern-Ricci flow starting at initial invariant metrics.
\end{prop*}

Throughout this note, we give a complete description of harmonic forms on such compact complex surfaces depending on the invariant Hermitian metrics. We made computations with the aid of SageMath \cite{sage}.

We ask whether for Dolbeault and Bott-Chern geometric formalities there exist obstructions (such as the ones found in \cite{kotschick}) which would help complete the picture for geometric formalities for class VII surfaces. We also ask whether the behaviour we observed is more general or there exist counterexamples.

\bigskip

\noindent{\small {\itshape Acknowledgments.}
This work has been originally developed as partial fulfillment for the second author's Master Degree in Matematica at Universit\`{a} di Firenze under the supervision of the first author.
The authors would like to thank Andrea Cattaneo, José Manuel Moreno-Fern\'andez, Cosimo Flavi, Nicoletta Tardini, and Adriano Tomassini for several interesting discussions.
Thanks also to the anonymous Referee for her/his comments and suggestions.
}

\section{Preliminaries on Hermitian manifolds}

\subsection*{Cohomologies of compact complex manifolds}
For a complex compact manifold $X$, one can define its real and complex de Rham cohomology groups by
\begin{equation*}
H_{dR}^{k}(X;\R):=\frac{\ker (d\colon\bigwedge^{k}(X)\to\bigwedge^{k+1}(X))}{\im (d\colon\bigwedge^{k-1}(X)\to\bigwedge^{k}(X))}\qquad H_{dR}^{k}(X;\C):=\frac{\ker(d\colon\bigwedge_{\C}^{k}(X)\to\bigwedge_{\C}^{k+1}(X))}{\im(d\colon\bigwedge_{\C}^{k-1}(X)\to\bigwedge_{\C}^{k}(X))}
\end{equation*}
which correspond, by the universal coefficients theorem.

In the case of non-K\"ahler complex compact manifolds, further invariants regarding the complex structure are given by other cohomologies, such as Dolbeault, Bott-Chern and Aeppli cohomologies, defined respectively by
\begin{equation*}
H_{\delbar}^{\bullet,\bullet}(X):=\frac{\ker\delbar}{\im\delbar}\qquad H_{BC}^{\bullet,\bullet}(X):=\frac{\ker\del\cap\ker\delbar}{\im(\del\delbar)}\qquad H_A^{\bullet,\bullet}(X):=\frac{\ker\del\delbar}{\im\del +\im\delbar}.
\end{equation*}

In general, the homomorphisms induced by the identity between those cohomologies are not necessarily injective nor surjective. However, for compact manifolds satisfying the $\del\delbar$-lemma, such as compact K\"ahler manifolds, these natural maps that are induced by the identity are all isomorphisms, see \cite{deligne-griffiths-morgan-sullivan}:
\[
\xymatrix{ 
& \ar[dl] H_{BC}^{\bullet,\bullet}(X) \ar[dr] \ar[d] &\\
H_{\delbar}^{\bullet,\bullet}(X) \ar[dr] & H_{dR}^{\bullet}(X;\C) \ar[d] & H_{\del}^{\bullet,\bullet}(X)\ar[ld]\\
& H_A^{\bullet,\bullet}(X)&}
\]
where $H_{\del}^{\bullet,\bullet}(X)$ is the conjugate of the Dolbeault cohomology.

Once fixed a Hermitian metric $g$ on $X$, the Hodge-$\ast$-operator with respect to such metric induces isomorphisms, see {\itshape e.g.} \cite[§2.b, §2.c]{schweitzer}:
\begin{equation}\label{eq:hdg-iso}
H_{dR}^{k}(X) \xrightarrow{\sim}  H_{dR}^{2n-k}(X),\quad H_{\delbar}^{p,q}(X)\xrightarrow{\sim} H_{\del}^{n-q,n-p}(X),\quad H_{BC}^{p,q}(X)\xrightarrow{\sim} H_{A}^{n-q,n-p}(X),
\end{equation}
where $\dim_{\R}X=2n$, $1\leq k \leq 2n$, and $1\leq p,q\leq n$.

\subsection*{Hodge theory}
For a compact Hermitian manifold $(X^n,g)$, we denote by $\mathcal{H}_{\openbox}^{k}(X, g)$, with $k\in\{1,\dots,2n\}$, $\openbox\in\{dR,\delbar,BC,A\}$ the set of harmonic $k$-forms with respect to the regular Laplacian, Dolbeault Laplacian, Bott-Chern Laplacian, Aeppli Laplacian, namely the following self-adjoint, elliptic operators, see \cite[§2.b, §2.c]{schweitzer}:
\begin{align*}
&\Delta_{dR}=dd^*+d^*d, \qquad\qquad\Delta_{\delbar}=\delbar\,\delbar^*+\delbar^*\delbar\\
&\Delta_{BC}=\del\delbar\,\delbar^*\del^* + \delbar^*\del^*\del\delbar + \delbar^*\del\del^*\delbar + \del^*\delbar\,\delbar^*\del + \delbar^*\delbar+\del^*\del\\
&\Delta_{A}=\del\del^* + \delbar\,\delbar^* + \delbar^*\del^*\del\delbar + \del\delbar\,\delbar^*\del^* + \del\delbar^*\del\del^* + \delbar\del^*\del\delbar^*.
\end{align*}
Clearly, those operators depend on the fixed metric $g$.
For $\openbox\in\{\delbar,BC,A\}$, we can define also $\mathcal{H}_{\openbox}^{p,q}(X,g)$, \emph{i.e.} the set of harmonic $(p,q)$-forms with respect to Dolbeault, Bott-Chern and Aeppli Laplacians.

Both the de Rham $(\bigwedge_{\C}^{\bullet}(X),d)$ and Dolbeault complexes $(\bigwedge^{\bullet,\bullet}(X),\del,\delbar)$ of a compact Hermitian manifold $(X,g)$, via harmonic forms, admit decompositions
that yield the following isomorphisms of vector spaces between harmonic forms and cohomology classes:
\begin{equation}\label{eq:hodge-theory}
\begin{array}{rcl}
\mathcal{H}_{dR}^k(X,g)\xrightarrow{\sim} H_{dR}^k(X), &\qquad& \mathcal{H}_{\delbar}^{p,q}(X,g)\xrightarrow{\sim} H_{\delbar}^{p,q}(X),\\
\mathcal{H}_{BC}^{p,q}(X,g)\xrightarrow{\sim} H_{BC}^{p,q}(X), &\qquad& \mathcal{H}_A^{p,q}(X,g)\xrightarrow{\sim} H_A^{p,q}(X).
\end{array}
\end{equation}
We note that isomorfisms as in (\ref{eq:hdg-iso}) still hold at the level of harmonic forms.
\subsection*{Geometric formalities}
In general, whereas de Rham, Dolbeault and Bott-Chern cohomologies have a structure of algebra induced by the cup product of cohomology classes, the mentioned isomorphisms (\ref{eq:hodge-theory}) do not preserve the structure of algebra, expected by the wedge product. For Aeppli cohomology, especially, we do not have a structure of algebra \emph{a priori}, but a structure of $H_{BC}(X)$-module; the isomorphism in (\ref{eq:hodge-theory}) could not preserve this structure. This remark makes non-trivial the following definitions.
\begin{defi}
Let $X$ be a compact complex manifold, $g$ a fixed Hermitian metric on $X$. The metric $g$ is said to be
\begin{enumerate}
\item \emph{geometrically formal according to Kotschick} if $\mathcal{H}_{dR}^{\bullet}(X,g)$ has a structure of algebra induced by the wedge product;
\item \emph{Dolbeault-geometrically formal} if $\mathcal{H}_{\delbar}^{\bullet,\bullet}(X,g)$ has a structure of algebra induced by the wedge product;
\item \emph{Bott-Chern-geometrically formal} if $\mathcal{H}_{BC}^{\bullet,\bullet}(X,g)$ has a structure of algebra induced by the wedge product.
\end{enumerate}
\end{defi}
\subsection*{Chern-Ricci flow}
The Chern-Ricci flow (introduced in \cite{gill} and studied in \cite{tosatti-weinkove-JDG}) is a parabolic geometric flow that preserves the Hermitian condition of the initial given metric.
The equations that describe such flow on a Hermitian manifold $(X^n,J,g_0)$ are
\begin{equation*}
\frac{\del}{\del t}\omega(t)=-\text{Ric}^{Ch}(\omega(t)), \qquad\omega(0)=\omega_0,
\end{equation*}
where $\omega_0,\,\omega(t)$ are the foundamental forms associated, respectively, to the Hermitian initial metric $g_0$ and the evolution metric $g(t)$ by the usual relation $\omega(\cdot,\cdot)=g(J(\cdot),\cdot)$. For an arbitrary real $(1,1)$-form $\omega$, $\text{Ric}^{Ch}(\omega)$ is the Chern-Ricci form of $\omega$. The first Chern-Ricci curvature $\text{Ric}^{Ch}$ is defined starting from $\nabla^{Ch}$, the Chern connection on $(X,J,g)$, \emph{i.e.} the unique connection $\nabla$ on the holomorphic tangent bundle $T^{1,0}X$ such that $\nabla$ is compatible with both $g$ and $J$ and $\nabla^{0,1}=\delbar$.
In a holomorphic chart, the curvature tensor $R^{Ch}$ of such connection has components $R_{i\overline{j}k\overline{l}}^{Ch}$, for $i,j,k,l\in\{1,\dots,n\}$. The Chern-Ricci tensor is obtained by contracting the last two indices via the metric
\begin{equation*}
\text{Ricci}_{i\overline{j}}^{Ch}:=g^{k\overline{l}}R_{i\overline{j}k\overline{l}}^{Ch},
\end{equation*}
where $(g^{k\overline{l}})$ is the inverse of the matrix $(g_{k\overline{l}})$ representing in local coordinates the metric $g$. The Chern-Ricci form is defined by
\begin{equation*}
\text{Ric}^{Ch}:=\text{Ricci}^{Ch}(J(\,\cdot\,),\,\cdot\,).
\end{equation*}
Such form has important properties, among which a very simple form in local coordinates:
\begin{equation*}
\text{Ric}^{Ch}(\omega)=-\sqrt{-1}\del\delbar\log\det(g),
\end{equation*}
from which we can deduce that $\text{Ric}^{Ch}(\omega)$ is a $\del$-, $\delbar$-closed form, hence it defines a cohomology class in $H_{BC}^{1,1}(X)$. Such class is a holomorphic invariant, denoted by $c_1^{BC}(X)$, which plays a fundamental role in the classification of complex manifolds.

\section{Preliminaries on compact complex surfaces and quotients of Lie groups}
In this section, we analyze in details complex structures, cohomologies, and Chern-Ricci flow on non-K\"ahler compact complex surfaces that can be described as quotients of Lie groups $G$ endowed with invariant complex structure \cite{hasegawa-jsg}, namely Hopf, Inoue, and Kodaira surfaces. Here, {\em invariant} tensors are meant to be defined on the covering Lie group and invariant by the action of the Lie group on itself by left-translations, that is, they can be considered as linear tensors over the associated Lie algebra $\mathfrak g$.

\subsection*{Complex structure}
More precisely, we describe the complex structure $J$ by a coframe of invariant $(1,0)$-forms $\{\varphi^1,\varphi^2\}$ and their conjugates, and by their structure equations
$$ d\varphi^I=-c_{HK}^I \varphi^H\wedge\varphi^K , $$
equivalently, by the dual frame $\{\varphi_1,\varphi_2\}$ of $(1,0)$-vector fields and their conjugates, with structure equations $[\varphi_H,\varphi_K]=c_{HK}^I \varphi_I$.
(Capital letters here vary in the ordered set $(1,2,\bar1,\bar2)$ and refer to the corresponding component. The Einstein summation is assumed, for increasing indices in case of forms. Hereafter, we shorten {\itshape e.g.} $\varphi^{2\bar1}:=\varphi^2\wedge\bar\varphi^1$.)

\subsection*{Hermitian structure}
The arbitrary invariant Hermitian metric $g:=\omega(\,\cdot\,,J(\cdot))$ has associated $(1,1)$-form
\begin{equation}\label{eq:metric}
2\,\omega = \sqrt{-1}\sum_{I,\,J\,=1}^2 g_{I\overline{J}}\,\varphi^I\wedge\varphi^{\overline{J}} = \sqrt{-1} r^2\, \varphi^{1\bar1} + \sqrt{-1} s^2\,\varphi^{2\bar2} + u\, \varphi^{1\bar2} - \bar u\, \varphi^{2\bar1}
\end{equation}
where the coefficients satisfy
$$ r^2>0 ,\qquad s^2>0 , \qquad r^2s^2 > |u|^2 .$$
That is to say, the Hermitian matrix
$$ (g_{KL})_{K,L} = \frac{1}{2}\cdot\left(\begin{matrix}
r^2 & -\sqrt{-1}u \\
\sqrt{-1}{\bar u} & {s^2}
\end{matrix}\right) \in \mathrm{GL}(\mathfrak g) $$
is positive-definite. Its inverse is
$$ (g^{KL})_{K,L} := (g_{KL})_{K,L}^{-1} =
\frac{2}{r^2s^2-|u|^2}\cdot
\left(\begin{matrix}
s^2 & \sqrt{-1}u \\
-\sqrt{-1}\bar u & r^2
\end{matrix}\right)
.$$

The Christoffel symbols of the Chern connection can be computed as follows, see {\itshape e.g.} \cite{OUV-strominger}:
$$ (\Gamma^{Ch})_{IH}^{K} = \frac{1}{2}\, c_{IH}^{K} - \frac{1}{2}\, g^{KA}g_{BI}c_{HA}^{B} - \frac{1}{2}\, g^{KA}g_{BH}c_{IA}^{B} + \frac{1}{2} g^{KL} C_{IHL} , $$
where
$$ C_{IHL} = d\omega(J\varphi_I, \varphi_H, \varphi_L) . $$
We can then express the $(4,0)$-Riemannian curvature of the Chern connection as
$$ (R^{Ch})_{IHKL} = g_{AL}(\Gamma^{Ch})_{HK}^{B}(\Gamma^{Ch})_{IB}^{A} - g_{AL} (\Gamma^{Ch})_{IK}^{B} (\Gamma^{Ch})_{HB}^{A} - g_{AL} c_{IH}^{B} (\Gamma^{Ch})_{BK}^{A} , $$
and the Chern-Ricci tensor as
$$ (\mathrm{Ricci}^{Ch})_{IH} = g^{KL} (R^{Ch})_{IHKL} . $$
Then the Chern-Ricci form is
$$ \mathrm{Ric}^{Ch} = \mathrm{Ricci}^{Ch}(J(\,\cdot\,), \,\cdot\,) \in c_1^{BC}(X) \in H^{1,1}_{BC}(X;\mathbb R). $$

Finally, we collect here some explicit description of the Hodge-star-operator on forms for the arbitrary Hermitian metric associated to the form \eqref{eq:metric}, in order to describe harmonicity, see also \cite[Lemma 2]{picard-cime}. It is straightforward to check that:
\begin{eqnarray}
\ast_g\varphi^{1} &=& \frac{\sqrt{-1}}{2} \, \overline{u} \varphi^{12\bar1} + \frac{1}{2} \, s^{2}  \varphi^{12\bar2} , \label{eq:star-10} \qquad 
\ast_g\varphi^{2} = -\frac{1}{2} \, r^{2} \varphi^{12\bar1} + \frac{\sqrt{-1}}{2} \, u  \varphi^{12\bar2} , \\
\ast_g\bar\varphi^{1} &=& -\frac{\sqrt{-1}}{2} \, u \varphi^{1\bar1\bar2} + \frac{1}{2} \, s^{2}  \varphi^{2\bar1\bar2} , \label{eq:star-01}  \qquad
\ast_g\bar\varphi^{2} = -\frac{1}{2} \, r^{2} \varphi^{1\bar1\bar2} - \frac{\sqrt{-1}}{2} \, \overline{u}  \varphi^{2\bar1\bar2} ; \\
\ast_g\varphi^{12} &=& \varphi^{12}, \label{eq:star-20} \qquad
\ast_g\varphi^{\bar1\bar2} = \varphi^{\bar1\bar2}
\\
V\ast_g\varphi^{1\overline{1}}&=&|u|^2\varphi^{1\overline{1}}-\sqrt{-1}us^2\varphi^{1\overline{2}}+\sqrt{-1}\overline{u}s^2\varphi^{2\overline{1}}+s^4\varphi^{2\overline{2}} , \label{eq:star-11} \\
V\ast_g\varphi^{1\overline{2}}&=&-\sqrt{-1}\overline{u}r^2\varphi^{1\overline{1}}-r^2s^2\varphi^{1\overline{2}}+\overline{u}^2\varphi^{2\overline{1}}-\sqrt{-1}\overline{u}s^2\varphi^{2\overline{2}}\nonumber\\
V\ast_g\varphi^{2\overline{1}}&=&\sqrt{-1}ur^2\varphi^{1\overline{1}}+u^2\varphi^{1\overline{2}}-r^2s^2\varphi^{2\overline{1}}+\sqrt{-1}us^2\varphi^{2\overline{2}}\nonumber\\
V\ast_g\varphi^{2\overline{2}}&=&r^4\varphi^{1\overline{1}}-\sqrt{-1}ur^2\varphi^{1\overline{2}}+\sqrt{-1}\overline{u}r^2\varphi^{2\overline{1}}+|u|^2\varphi^{2\overline{2}}\nonumber,\\
V \ast_g\varphi^{12\bar1} &=& -2 \sqrt{-1} \, u \varphi^{1} + 2 \, s^{2} \varphi^{2} , \label{eq:star-21} \qquad
V \ast_g\varphi^{12\bar2} = -2 \, r^{2} \varphi^{1} - 2 \sqrt{-1} \, \overline{u} \varphi^{2}, \\
V \ast_g\varphi^{1\bar1\bar2} &=& 2 \sqrt{-1} \, \overline{u} \bar\varphi^1  + 2 \, s^{2} \bar\varphi^{2} , \label{eq:star-12} \qquad
V \ast_g\varphi^{2\bar1\bar2} = -2 \, r^{2} \bar\varphi^{1} + 2 \sqrt{-1} \, u \bar\varphi^{2}; \nonumber
\end{eqnarray}
where we set $V={g^{1\overline{1}}g^{2\overline{2}}-g^{1\overline{2}}g^{2\overline{1}}} = r^2s^2-|u|^2$.

\subsection*{Cohomologies}
Consider the inclusion of invariant forms into the double complex of forms, $\iota\colon (\wedge^{\bullet,\bullet}\mathfrak g^\vee, \partial, \overline\partial) \hookrightarrow (\wedge^{\bullet,\bullet}X , \partial, \overline\partial)$. By choosing an invariant Hermitian metric, (the easier finite-dimensional version of) elliptic Hodge theory also applies at the level of invariant forms; in particular, any cohomology class of invariant forms admits a unique invariant harmonic representative. It follows that the above inclusion induces injective maps in de Rham $\iota_{dR}$, Dolbeault $\iota_{\overline\partial}$, Bott-Chern $\iota_{BC}$, Aeppli $\iota_{A}$ cohomology, see \cite[Lemma 9]{console-fino}. We claim that they are in fact isomorphisms.

The de Rham cohomology of Hopf, Inoue, Kodaira surfaces is well known, and it happens that the above maps $\iota_{dR}$ are actually isomorphisms, that is, any de Rham class admits an invariant representative. In fact, the Hopf surface is diffeomorphic to the product $\mathbb S^1\times \mathbb S^3$ of two compact Lie groups, so one can use the K\"unneth formula and {\itshape e.g.} \cite[Theorem 1.28]{felix-oprea-tanre}; the primary Kodaira surface is a nilmanifold, so one can use the Nomizu theorem \cite{nomizu}; the secondary Kodaira surfaces are quotients of primary Kodaira surfaces by finite groups; the Inoue surface of type $S^{\pm}$ is a completely-solvable solvmanifold, so one can use the Hattori theorem \cite{hattori}; and the de Rham cohomology of the Inoue surface of type $S_M$ can be computed by exploiting their number-theoretic construction as \cite{oeljeklaus-toma} does in a more general setting.

As for the Dolbeault cohomology, for compact complex surfaces, we know that the Fr\"olicher spectral sequence degenerates at the first page, see {\itshape e.g.} \cite{barth-hulek-peters-vandeven}, that is, $\dim_{\mathbb C} H^k_{dR}(X;\mathbb C)=\sum_{p+q=k} \dim_{\mathbb C} H^{p,q}_{\overline\partial}(X)$ for any $k$. By explicitly computing the Dolbeault cohomology of invariant forms \cite{angella-dloussky-tomassini}, one then notice that the above maps $\iota_{\overline\partial}$ are actually isomorphisms.

Finally, Bott-Chern cohomology of compact complex surfaces is well-undestood since \cite{teleman}. By \cite[Theorem 1.3, Proposition 2.2]{angella-kasuya-1}, (that fits in the general theory later developed by \cite{stelzig},) also $\iota_{BC}$ are isomorphisms. Explicit computations can be found in \cite{angella-dloussky-tomassini}. Finally, $\iota_A$ are isomorphisms thanks to the Schweitzer duality between Bott-Chern and Aeppli cohomologies, where one can use the Hodge-star-operator with respect to an invariant Hermitian metric.

Finally, by uniqueness of the harmonic representative in a cohomology class, we also deduce that harmonic representatives with respect to invariant metrics are invariant.

\subsection*{Chern-Ricci flow}
Recall that the Chern-Ricci form represents the first Bott-Chern class $c_1^{BC}(X) \in H^{1,1}_{BC}(X)$. Since a class in $H^{1,1}_{BC}(X)$ contains only one invariant representative, the Chern-Ricci form $\mathrm{Ric}^{Ch}(\omega)$ does not depend on the invariant Hermitian metric $\omega$. In particular, the Chern-Ricci flow starting at $\omega_0$ reduces to
\begin{equation}\label{eq:CRF}\tag{CRF}
 \frac{\partial}{\partial t}\omega(t) = -\mathrm{Ric}^{Ch}(\omega_0), \qquad \omega(0)=\omega_0 .
\end{equation}

We notice that the solution of the Chern-Ricci flow starting at an invariant metric remains invariant for any existence time. Indeed, by short existence and uniqueness assured by parabolicity, the symmetry group is preserved along the flow (and possibly increases in the limit, see {\itshape e.g.} \cite{lauret}).

Denote by $\rho_r$, $\rho_s$, $\rho_u$ the coefficients of the Chern-Ricci form,
$$ 2\mathrm{Ric}^{Ch}=\sqrt{-1}\rho_r \varphi^{1\bar1}+\sqrt{-1}\rho_s\varphi^{2\bar2}+\rho_u\varphi^{1\bar2}-\bar \rho_u\varphi^{2\bar1} , $$
and let the initial metric $\omega_0$ be of the form
\begin{equation}\label{eq:metric-initial}
2\,\omega_0 = \sqrt{-1} r_0^2\, \varphi^{1\bar1} + \sqrt{-1} s_0^2\,\varphi^{2\bar2} + u_0\, \varphi^{1\bar2} - \bar u_0\, \varphi^{2\bar1} ,
\end{equation}
where $r_0, s_0\in\mathbb R\setminus\{0\}$ and $u_0\in\mathbb C$ such that $r_0^2s_0^2-|u_0|^2>0$. The solution $\omega(t)$ of the Chern-Ricci flow starting at $\omega_0$ is then
$$ 2\omega(t) = \sqrt{-1} (r_0^2-t\rho_r)\varphi^{1\bar1} + \sqrt{-1}(s_0^2-t\rho_s) \varphi^{2\bar2} + (u_0-t\rho_u) \varphi^{1\bar2} - (\bar u_0-t\bar\rho_u) \varphi^{2\bar1} , $$
defined for times $t$ such that $r(t)^2=r_0^2-t\rho_r>0$, $s(t)=s_0^2-t\rho_s>0$, $u(t)=u_0-t\rho_u\in \mathbb C$ such that $r(t)^2s(t)^2-|u(t)|^2>0$.

\section{Geometric formality according to Kotshick}
In this section we state the main theorem of this note, regarding class VII surfaces of the Enriques-Kodaira classification of compact complex surfaces. 

\begin{thm}\label{thm:main}
On class VII surfaces of the Enriques-Kodaira classification, geometric formality according to Kotshick is preserved by the Chern-Ricci flow starting at initial invariant Hermitian metrics.
\end{thm}

\begin{proof}
Let $X$ be a class VII surface, that is, $\text{Kod}(X)=-\infty$ and $b_1(X)=1$. By \cite[Theorem 6]{kotschick}, for a compact oriented
Kotschick-geometrically formal 4-manifold $X$, the first Betti number satisfies $b_1(X)\in\{0,1,2,4\}$. Since all class VII surfaces are non-K\"ahler, they must have odd first Betti number by \cite{buchdahl,lamari}, that is, $b_1(X)=1$. By \cite[Theorem 9]{kotschick}, the Euler
characteristic of such manifolds vanishes, implying that $b_2(X)=0$. Since the characterization result by \cite{bogomolov,kodaira-1,li-yau-zheng, teleman-ijm}, class VII surfaces with $b_2(X)=0$ are necessarily Hopf or Inoue surfaces, then we see that the only Kotschick-geometrically formal class VII surfaces can be Hopf and Inoue surfaces. Therefore, Chern-Ricci flow starting at any metric cannot produce geometrically formal metrics on class VII surfaces other than Hopf and Inoue surfaces: we will then check the statement for those surfaces.
\subsection*{Case 1: Hopf surfaces} Hopf surfaces $X$ are compact complex surfaces in class VII defined as a quotient of $\mathbb C^2 \setminus \{0\}$ by a free action of a discrete group generated by a holomorphic contraction $\gamma(z,w)=(\alpha z + \lambda w^p, \beta w)$ where $\alpha,\beta,\lambda\in\mathbb C$ and $p\in\mathbb N$ are such that $0<|\alpha|\leq|\beta|<1$ and $(\alpha-\beta^p)\lambda=0$, see \cite{lattes}, \cite[page 820]{sternberg}, see \cite[Remark 1]{wehler}.

The diffeomorphism type is $\mathbb S^1\times \mathrm{SU(2)}$, and the complex structure is a special case of the Calabi-Eckmann complex structure on product of spheres \cite{calabi-eckmann}. See also \cite[Theorem 4.1]{parton-rend}. In terms of a coframe $(\varphi^1,\varphi^2)$ of $(1,0)$-forms, they are described as
$$ d\varphi^1 = \sqrt{-1}\varphi^1\wedge\varphi^2+\sqrt{-1}\varphi^1\wedge\bar\varphi^2, \qquad d\varphi^2=-\sqrt{-1}\varphi^1\wedge\bar\varphi^1 .$$
The de Rham cohomology of Hopf surfaces is
\begin{eqnarray*}
H^\bullet_{dR}(X;\mathbb C) &=& \mathbb C \langle 1 \rangle \oplus \mathbb C \langle \varphi^2-\varphi^{\bar2} \rangle \oplus \mathbb C \langle \varphi^{12\bar1}-\varphi^{1\bar1\bar2} \rangle \oplus \mathbb C \langle \varphi^{12\bar1\bar2} \rangle
\end{eqnarray*}
where we have listed the harmonic representatives with respect to the Hermitian metric with fundamental form $2\omega=\sqrt{-1}\varphi^{1\bar1}+\sqrt{-1}\varphi^{2\bar2}$ instead of their classes.

\subsubsection*{Harmonic representatives for cohomologies}
We look at how harmonic representatives of de Rham cohomology change with respect to the invariant Hermitian metric, and in particular whether their product is still harmonic.

We notice that, varying the invariant Hermitian metric, the harmonic representatives are
$$ H^{\bullet}_{dR}(X;\mathbb R) = \mathbb C \langle 1 \rangle \oplus \mathbb C \langle \varphi^2-\varphi^{\bar2} \rangle \oplus \mathbb C \left\langle
-\frac{1}{2} \, r^{2}  \varphi^{12\bar1} + \frac{\sqrt{-1}}{2} \, u  \varphi^{12\bar2} + \frac{1}{2} \, r^{2} \varphi^{1\bar1\bar2} + \frac{\sqrt{-1}}{2} \overline{u}  \varphi^{2\bar1\bar2}
\right\rangle \oplus \mathbb C \langle \varphi^{12\bar1\bar2} \rangle . $$
Indeed, it suffices to check that the harmonic representative of the class $[\varphi^2-\bar\varphi^2]$ does not depend on the invariant metric. This is because harmonic representatives are invariant, and the class $[\varphi^2-\bar\varphi^2]=\{\varphi^2-\bar\varphi^2+dc \;:\; c \in \mathbb R\}$ contains only one invariant representantive, which is then harmonic with respect to any metric. Then we compute the harmonic representative of the dual class in $H^3_{dR}(X;\mathbb R)$ by applying the Hodge-star-operator to $\varphi^2-\bar\varphi^2$ with respect to the arbitrary Hermitian metric. In any case, the product of an invariant $1$-form and an invariant $3$-form is either zero or a scalar multiple of the volume form. It follows that {\em any invariant metric on the Hopf surface is geometrically formal in the sense of Kotschick}.

\subsubsection*{Chern-Ricci flow}
Clearly, on the Hopf surface with invariant Hermitian metrics, the properties of geometric formality in the sense of Kotschick is preserved along the Chern-Ricci flow.
Nonetheless, for completeness and for later use, we compute the Chern-Ricci form and the Chern-Ricci flow on $X$.

We start by computing the Chern-Riemann curvature of an invariant Hermitian metric.
We follow notation as in \cite[Section 2]{OUV-strominger} (see also \cite[Section 6]{liu-yang-IJM} for another argument).
With respect to the frame $\left( \varphi_1, \varphi_2, \bar\varphi_1, \bar\varphi_2 \right)$ and to the dual coframe $\left( \varphi^1, \varphi^2, \bar\varphi^1, \bar\varphi^2 \right)$, we set the structure constants
$$ [ \varphi_I, \varphi_H ] =: c_{IH}^K \varphi_K .$$
Here, capital letters vary in $\{1,2,\bar1,\bar2\}$, and the Einstein summation is assumed.
In our case, the non-trivial structure constants are
\begin{equation*}
\begin{array}{cccc}
c_{1 2}^{1} = -\sqrt{-1}, &
c_{1 \bar1}^{2} = \sqrt{-1}, &
c_{1 \bar1}^{\bar2} = \sqrt{-1}, &
c_{1 \bar2}^{1} = -\sqrt{-1}, \\
c_{2 1}^{1} = \sqrt{-1}, &
c_{2 \bar1}^{\bar1} = -\sqrt{-1}, &
c_{\bar1 1}^{2} = -\sqrt{-1}, &
c_{\bar1 1}^{\bar 2} = -\sqrt{-1}, \\
c_{\bar1 2}^{\bar1} = \sqrt{-1}, &
c_{\bar1 \bar2}^{\bar1} = \sqrt{-1}, &
c_{\bar2 1}^{1} = \sqrt{-1}, &
c_{\bar2 \bar1}^{\bar1} = -\sqrt{-1} .
\end{array}
\end{equation*}
Recall that the Christoffel symbols of the Levi-Civita connections (with respect to the above non-commutative frame) can be computed as
\begin{eqnarray*}
(\Gamma^{LC})_{IH}^{K} &=& \frac{1}{2}\, g^{KL}\, \left( g([\varphi_I,\varphi_H], \varphi_L) - g([\varphi_H,\varphi_L], \varphi_I) - g([\varphi_I,\varphi_L], \varphi_H) \right) \\[5pt]
&=& \frac{1}{2}\, c_{IH}^{K} - \frac{1}{2}\, g^{KA}g_{BI}c_{HA}^{B} - \frac{1}{2}\, g^{KA}g_{BH}c_{IA}^{B} .
\end{eqnarray*}
Set $V=r^2s^2 - |u|^2$ for simplicity.
In our case, up to conjugation, the non-trivial ones are
\begin{equation*}
\begin{array}{cc}
(\Gamma^{LC})_{11}^{1} = -s^2uV^{-1}, &
(\Gamma^{LC})_{11}^{2} = -\sqrt{-1}u^2V^{-1}, \\
(\Gamma^{LC})_{12}^{1} = \frac{1}{2}(-\sqrt{-1}s^4 + \sqrt{-1}|u|^2)V^{-1}, &
(\Gamma^{LC})_{12}^{2} = -\frac{1}{2}(r^2 - s^2)uV^{-1} , \\
(\Gamma^{LC})_{1\bar1}^{1} = \frac{1}{2}s^2\bar uV^{-1} , &
(\Gamma^{LC})_{1\bar1}^{2} = \frac{1}{2}\sqrt{-1}r^2s^2V^{-1} , \\
(\Gamma^{LC})_{1\bar1}^{\bar1} = \frac{1}{2}s^2 u V^{-1} , &
(\Gamma^{LC})_{1\bar1}^{\bar2} = \frac{1}{2} (\sqrt{-1} r^2 s^2 - 2\sqrt{-1}|u|^2)V^{-1} , \\
(\Gamma^{LC})_{1\bar2}^{1} = -\frac{1}{2}\sqrt{-1}s^4V^{-1} , &
(\Gamma^{LC})_{1\bar2}^{2} = \frac{1}{2} s^2uV^{-1} , \\
(\Gamma^{LC})_{1\bar2}^{\bar1} = \frac{1}{2} \sqrt{-1}u^2V^{-1} , &
(\Gamma^{LC})_{1\bar2}^{\bar2} = \frac{1}{2}r^2uV^{-1} , \\
(\Gamma^{LC})_{21}^{1} = \frac{1}{2} (2\sqrt{-1} r^2s^2 - \sqrt{-1}s^4 - \sqrt{-1}|u|^2)V^{-1}, &
(\Gamma^{LC})_{21}^{2} = -\frac{1}{2}(r^2 - s^2)u V^{-1}  , \\
(\Gamma^{LC})_{22}^{1} = -s^2 \bar u V^{-1}, &
(\Gamma^{LC})_{22}^{2} = -\sqrt{-1}|u|^2 V^{-1}, \\
(\Gamma^{LC})_{2\bar1}^{1} = -\frac{1}{2}\sqrt{-1} \bar u^2V^{-1} , &
(\Gamma^{LC})_{2\bar1}^{2} = \frac{1}{2}r^2 \bar u V^{-1} , \\
(\Gamma^{LC})_{2\bar1}^{\bar1} = \frac{1}{2}(-2*\sqrt{-1}r^2s^2 + \sqrt{-1}s^4 + 2\sqrt{-1}|u|^2)V^{-1} , &
(\Gamma^{LC})_{2\bar1}^{\bar2} = \frac{1}{2}s^2 \bar u V^{-1} , \\
(\Gamma^{LC})_{2\bar2}^{1} = -\frac{1}{2} s^2 \bar u V^{-1} , &
(\Gamma^{LC})_{2\bar2}^{2} = -\frac{1}{2}\sqrt{-1}|u|^2V^{-1} , \\
(\Gamma^{LC})_{2\bar2}^{\bar1} = -\frac{1}{2}s^2uV^{-1}, &
(\Gamma^{LC})_{2\bar2}^{\bar2} = \frac{1}{2}\sqrt{-1}|u|^2V^{-1} .
\end{array}
\end{equation*}
We can now compute the Christoffel symbols $(\Gamma^{Ch})_{IH}^{K}$ of the Chern connection by the formula \cite[Equation (7)]{OUV-strominger}:
$$ (\Gamma^{\varepsilon,\rho})_{IH}^{K} = (\Gamma^{LC})_{IH}^{K} + \varepsilon g^{KL} T_{IHL} + \rho g^{KL} C_{IHL} , $$
by setting $(\varepsilon,\rho)=\left(0,\frac{1}{2}\right)$, where
$$ T_{IHL} = -d\omega (J\varphi_I, J\varphi_H, J\varphi_L), \qquad C_{IHL} = d\omega(J\varphi_I, \varphi_H, \varphi_L) . $$
We get
\begin{equation*}
\begin{array}{cc}
(\Gamma^{Ch})_{21}^{2} = -r^2 uV^{-1} , &
(\Gamma^{Ch})_{21}^{1} = \sqrt{-1}r^2s^2V^{-1} , \\
(\Gamma^{Ch})_{1\bar1}^{\bar2} = \sqrt{-1} , &
(\Gamma^{Ch})_{2\bar1}^{\bar1} = -\sqrt{-1} , \\
(\Gamma^{Ch})_{12}^{1} = -\sqrt{-1}s^4V^{-1} , &
(\Gamma^{Ch})_{12}^{2} = s^2uV^{-1},
\end{array}
\end{equation*}
the others being equal to the corresponding Levi-Civita symbols or deduced by conjugation.
We can compute the $(4,0)$-Riemannian curvature of $\nabla^{\varepsilon,\rho}$ as
\begin{eqnarray*}
(R^{\varepsilon,\rho})_{IHKL} &=& g_{AL}(\Gamma^{\varepsilon,\rho})_{HK}^{B}(\Gamma^{\varepsilon,\rho})_{IB}^{A} - g_{AL} (\Gamma^{\varepsilon,\rho})_{IK}^{B} (\Gamma^{\varepsilon,\rho})_{HB}^{A} \\[5pt]
&& - g_{AL} c_{IH}^{B} (\Gamma^{\varepsilon,\rho})_{BK}^{A} .
\end{eqnarray*}
By using the symmetries for the Chern curvature $(R^{Ch})_{IHKL}=-(R^{Ch})_{HIKL}=-(R^{Ch})_{IHLK}$ and the conjugation, we get that the only non-zero components are
\begin{equation*}
\begin{array}{c}
(R^{Ch})_{1\bar11\bar1} = \frac{1}{2} (2r^4s^2 - r^2s^4 - 2(r^2 - s^2)|u|^2) V^{-1} , \\
(R^{Ch})_{1\bar11\bar2} = \frac{1}{2} (\sqrt{-1} |u|^2u + (-\sqrt{-1} r^2 s^2 - \sqrt{-1} s^4) u) V^{-1} , \\
(R^{Ch})_{1\bar12\bar1} = \frac{1}{2}(-\sqrt{-1}|u|^2\bar u - (-\sqrt{-1}r^2s^2 - \sqrt{-1}s^4)\bar u) V^{-1} , \\
(R^{Ch})_{1\bar12\bar2} = \frac{1}{2} s^6 V^{-1} , \\
(R^{Ch})_{1\bar21\bar1} = \frac{1}{2}(-\sqrt{-1}r^2s^2u + 2\sqrt{-1}|u|^2u) V^{-1} , \\
(R^{Ch})_{1\bar21\bar2} = \frac{1}{2} s^2u^2 V^{-1} , \\
(R^{Ch})_{1\bar12\bar1} = -\frac{1}{2} s^2 |u|^2 V^{-1} , \\
(R^{Ch})_{(1\bar22\bar2} = \frac{1}{2}\sqrt{-1}s^4 u V^{-1} , \\
(R^{Ch})_{2\bar11\bar1} = \frac{1}{2} (\sqrt{-1}r^2s^2\bar u - 2\sqrt{-1}|u|^2\bar u) V^{-1} , \\
(R^{Ch})_{2\bar11\bar2} = -\frac{1}{2} s^2 |u|^2 V^{-1} , \\
(R^{Ch})_{2\bar12\bar1} = \frac{1}{2}s^2\bar u^2 V^{-1} , \\
(R^{Ch})_{2\bar12\bar2} = -\frac{1}{2}\sqrt{-1} s^4 \bar u V^{-1} , \\
(R^{Ch})_{2\bar21\bar1} = \frac{1}{2} r^2|u|^2 V^{-1} , \\
(R^{Ch})_{2\bar21\bar2} = -\frac{1}{2}\sqrt{-1}r^2s^2u V^{-1} , \\
(R^{Ch})_{2\bar22\bar1} = \frac{1}{2}\sqrt{-1}r^2s^2\bar u V^{-1} , \\
(R^{Ch})_{2\bar22\bar2} = \frac{1}{2}s^2|u|^2 V^{-1} .
\end{array}
\end{equation*}
Finally, we can compute the (first) Chern-Ricci curvature by tracing on the third and fourth indices:
$$ (\mathrm{Ric}^{Ch})_{IH} = g^{KL}(R^{Ch})_{IHKL} ; $$
then the Cher-Ricci form can be defined as
$$ \mathrm{Ric}^{Ch} = (\mathrm{Ric}^{Ch})_{ih} \sqrt{-1} dz^i \wedge d\bar z^h . $$
In our case, the only non-trivial coefficients are
$$ (\mathrm{Ric}^{Ch})_{1\bar1} = 2 $$
and the corresponding $(\mathrm{Ric}^{Ch})_{\bar11}=-(\mathrm{Ric}^{Ch})_{1\bar1}$. Therefore the Chern-Ricci form of any invariant Hermitian metric is
$$ \mathrm{Ric}^{Ch}(\omega)=2\sqrt{-1} \varphi^1\wedge\bar\varphi^1 . $$ Therefore the solution of the Chern-Ricci flow starting at $\omega_0$ of the form \eqref{eq:metric-initial} is
\begin{equation}\label{eq:crf-hopf}
2\omega(t) = \sqrt{-1}(r_0^2-t)\varphi^{1\bar1}+\sqrt{-1}s_0^2\varphi^{2\bar2}+u_0\varphi^{1\bar2}-\bar u_0\varphi^{2\bar1},
\end{equation}
defined as long as $t<\frac{r_0^2s_0^2-|u_0|^2}{s_0^2}<r_0^2$.

\subsection*{Case 2: Inoue surfaces}
Inoue-Bombieri surfaces \cite{inoue,bombieri} $X$ are compact complex surfaces in class VII with second Betti number equal to zero and with no holomorphic curves \cite{bogomolov-1976, bogomolov, li-yau-zheng0, li-yau-zheng, teleman-ijm}.
Their universal cover is $\mathbb C\times\mathbb H$, where $\mathbb H$ denotes the upper half-plane.
They are divided into three families, $S_M$, $S^+_{N,p,q,r;\mathbf{t}}$, and $S^-_{N,p,q,r}$, depending on parameters.

\subsection*{Case 2.1: Inoue-Bombieri surface of type $S_M$}
We focus now on the case $S_M$: it has a structure of fibre bundle over $\mathbb{S}^1$, where the fibre is a $3$-dimensional torus.

Inoue-Bombieri surfaces of type $S_M$ admit a description as quotients of solvable Lie groups wih invariant complex structure \cite{hasegawa-jsg}, that we now describe.
We can fix a coframe $(\varphi^1,\varphi^2)$ of $(1,0)$-forms with structure equations
$$ d\varphi^1 = \frac{\alpha-\sqrt{-1}\beta}{2\sqrt{-1}}\varphi^{1}\wedge\varphi^2-\frac{\alpha-\sqrt{-1}\beta}{2\sqrt{-1}}\varphi^{1}\wedge\bar\varphi^2 , $$
$$ d\varphi^2=-\sqrt{-1}\alpha\varphi^2\wedge\bar\varphi^2 ,$$
where $\alpha\in\mathbb R\setminus\{0\}$ and $\beta\in\mathbb R$.
The de Rham cohomology of $X$ can be explicitly described \cite{teleman}, see \cite[Theorem 4.1]{angella-dloussky-tomassini}:
\begin{eqnarray*}
H^\bullet_{dR}(X;\mathbb R) &=& \mathbb C \langle 1 \rangle \oplus \mathbb C \langle \varphi^2-\varphi^{\bar2} \rangle \oplus \mathbb C \langle \varphi^{12\bar1}-\varphi^{1\bar1\bar2} \rangle \oplus \mathbb C \langle \varphi^{12\bar1\bar2} \rangle
\end{eqnarray*}
where we have listed the harmonic representatives with respect to the Hermitian metric with fundamental form $2\omega=\sqrt{-1}\varphi^{1\bar1}+\sqrt{-1}\varphi^{2\bar2}$ instead of their classes.

\subsubsection*{Harmonic representatives for cohomologies}
We list the harmonic representatives with respect to the arbitrary Hermitian metric as in \eqref{eq:metric}:
\begin{align}\label{eq:dr-SM}
&H_{dR}^{\bullet}(X;\R) = \\
&\mathbb C \langle 1 \rangle \oplus \mathbb C \langle \varphi^2-\varphi^{\bar2} \rangle \oplus \mathbb C \left\langle -\frac{1}{2} r^2 \varphi^{12\bar1} + \frac{\sqrt{-1}}{2}  u \varphi^{12\bar2} + \frac{1}{2} r^2 \varphi^{1\bar1\bar2} + \frac{\sqrt{-1}}{2} \overline{u} \varphi^{2\bar1\bar2} \right\rangle \oplus \mathbb C \langle \varphi^{12\bar1\bar2} \rangle\notag
\end{align}

We conclude that: {\em any invariant Hermitian metric on an Inoue surface of type $S_M$ is geometrically formal in the sense of Kotschick}.

\subsubsection*{Chern-Ricci flow}
The Chern-Ricci form of any invariant Hermitian metric is
$$ 2\mathrm{Ric}^{Ch}(\omega)=-\sqrt{-1}\alpha^2 \varphi^2\wedge\bar\varphi^2 , $$
whence the solution of the Chern-Ricci flow \eqref{eq:CRF} is given by
\begin{equation}\label{eq:crf-SM}
2\omega(t) = \sqrt{-1}r_0^2\varphi^{1\bar1}+\sqrt{-1}\left(s_0^2+\alpha^2t\right)\varphi^{2\bar2}+u_0\varphi^{1\bar2}-\bar u_0\varphi^{2\bar1},
\end{equation}
defined for any non-negative time $t\geq0$.

Clearly, on an Inoue surface of type $S_M$ with invariant Hermitian metrics, the properties of geometric formality in the sense of Kotschick is preserved along the Chern-Ricci flow.

\subsection*{Case 2.2: Inoue surfaces of class $S^{\pm}$}
In this subsection, we focus on the case of Inoue surfaces of type $S^\pm$.
Inoue-Bombieri surfaces of type $S^-$ have an unramified double cover of type $S^+$: we can then restrict to Inoue-Bombieri surfaces of type $S^+$, which have a structure of fibre bundle over $\mathbb S^1$, where the fibre is a compact quotient of the $3$-dimensional Heisenberg group.

Also Inoue-Bombieri surfaces of type $S^+$ admit a description as quotients of solvable Lie groups \cite{hasegawa-jsg}, that we now describe.
We can fix a coframe $(\varphi^1,\varphi^2)$ of $(1,0)$-forms with structure equations
$$ d\varphi^1 = \frac{1}{2\sqrt{-1}}\varphi^{1}\wedge\varphi^2+\frac{1}{2\sqrt{-1}}\varphi^{2}\wedge\bar\varphi^1+\frac{q\sqrt{-1}}{2}\varphi^2\wedge\bar\varphi^2 , $$
$$ d\varphi^2=\frac{1}{2\sqrt{-1}}\varphi^2\wedge\bar\varphi^2 ,$$
where $q\in\mathbb R$.
The de Rham cohomology of $X$ can be explicitly described \cite{teleman}, see \cite[Theorem 4.1]{angella-dloussky-tomassini}:
\begin{eqnarray*}
H^\bullet_{dR}(X;\mathbb C) &=& \mathbb C \langle 1 \rangle \oplus \mathbb C \langle \varphi^2-\varphi^{\bar2} \rangle \oplus \mathbb C \langle \varphi^{12\bar1}-\varphi^{1\bar1\bar2}\rangle \oplus \mathbb C \langle \varphi^{12\bar1\bar2} \rangle
\end{eqnarray*}
where we have listed the harmonic representatives with respect to the Hermitian metric with fundamental form $\omega=\sqrt{-1}\varphi^{1\bar1}+\sqrt{-1}\varphi^{2\bar2}$ instead of their classes.

\subsubsection*{Harmonic representatives for cohomologies}
The situation is exactly as in (\ref{eq:dr-SM}). We conclude that: {\em any invariant Hermitian metric on an Inoue surface of type $S^{\pm}$ is geometrically formal in the sense of Kotschick}.

\subsubsection*{Chern-Ricci flow}
The Chern-Ricci form of any invariant Hermitian metric is
$$ 2\mathrm{Ric}^{Ch}(\omega)=-\sqrt{-1}\varphi^{2\bar2} , $$
whence the solution of the Chern-Ricci flow \eqref{eq:CRF} is given by
\begin{equation}\label{eq:crf-Spm}
2\omega(t) = \sqrt{-1}r_0^2\varphi^{1\bar1}+\sqrt{-1}\left(s_0^2+t\right)\varphi^{2\bar2}+u_0\varphi^{1\bar2}-\bar u_0\varphi^{2\bar1},
\end{equation}
defined for any non-negative time $t\geq0$.

Clearly, on an Inoue surface of type $S^\pm$ with invariant Hermitian metrics, the properties of geometric formality in the sense of Kotschick is preserved along the Chern-Ricci flow.
\end{proof}

We also analyze in details primary and secondary Kodaira surfaces resulting in the following proposition, for which we give explicit computations.

\begin{prop}\label{prop:Kodaira-1}
On any Kodaira surface, the properties of geometric formality in the sense of Kotschick is preserved along the Chern-Ricci flow starting at initial invariant Hermitian metrics.
\end{prop}

\begin{proof}
We will look at each case separatedly.

\subsection*{Case 1: Primary Kodaira surface}
Kodaira surfaces $X$ are compact complex surfaces of Kodaira dimension $\text{Kod}(X)=0$ and first Betti number $b_1(X)=3$.
Primary Kodaira surfaces have trivial canonical bundle. 

We note that, by \cite[Theorem 6]{kotschick}, primary Kodaira surfaces are never Kotschick-geometrically formal, not even with regards to non-invariant metrics, by having $b_1=3$: hence \emph{Chern-Ricci flow preserves geometric formality according to Kotschick}.
The same conclusion follows by \cite[Theorem 1]{hasegawa-pams} stating that non-tori nilmanifolds are never formal, therefore never geometrically formal in the sense of Kotschick.
Nevertheless we give explicit computations for this fact.

We recall the description of primary Kodaira surfaces as quotients of solvable Lie groups \cite{hasegawa-jsg}.
There exists a coframe $(\varphi^1,\varphi^2)$ of $(1,0)$-forms with structure equations
$$ d\varphi^1 = 0 , \qquad d\varphi^2=\frac{\sqrt{-1}}{2}\varphi^1\wedge\bar\varphi^1 .$$
The de Rham cohomology of $X$ can be explicitly described:
\begin{eqnarray*}
H^\bullet_{dR}(X;\mathbb R) &=& \mathbb C \langle 1 \rangle \oplus \mathbb C \langle \varphi^{1}, \varphi^{\bar1}, \varphi^2-\varphi^{\bar2} \rangle \oplus \mathbb C \langle \varphi^{12}, \varphi^{1\bar2}, \varphi^{2\bar1}, \varphi^{\bar1\bar2} \rangle \\
&& \oplus \mathbb C \langle \varphi^{12\bar2}, \varphi^{2\bar1\bar2},\varphi^{12\bar1}-\varphi^{1\bar1\bar2} \rangle \oplus \mathbb C \langle \varphi^{12\bar1\bar2} \rangle
\end{eqnarray*}
where we have listed the harmonic representatives with respect to the Hermitian metric with fundamental form $\omega=\sqrt{-1}\varphi^{1\bar1}+\sqrt{-1}\varphi^{2\bar2}$ instead of their classes.

We list the harmonic representatives with respect to the arbitrary Hermitian metric as in \eqref{eq:metric}:
\begin{eqnarray*}
H^\bullet_{dR}(X;\mathbb R) &=& \mathbb C \langle 1 \rangle \oplus \mathbb C \langle \varphi^{1}, \varphi^{\bar1}, \varphi^2-\varphi^{\bar2} \rangle \oplus \mathbb C \left\langle \varphi^{12}, \varphi^{1\bar2}+\frac{\sqrt{-1} \, \overline{u}}{s^{2}}  \varphi^{1\bar1} , \varphi^{2\bar1} - \frac{\sqrt{-1} \, u}{s^{2}}  \varphi^{1\bar1}, \varphi^{\bar1\bar2} \right\rangle \\
&& \oplus \mathbb C \left\langle \frac{1}{2} \, s^{2}  \varphi^{12\bar2} + \frac{\sqrt{-1}}{2} \, \overline{u}  \varphi^{12\bar1} , \frac{1}{2} \, s^{2} \varphi^{2\bar1\bar2} - \frac{\sqrt{-1}}{2}\, u  \varphi^{1\bar1\bar2} , \right.\\
&& \left. -\frac{1}{2} \, r^{2}  \varphi^{12\bar1} + \frac{\sqrt{-1}}{2} \, u  \varphi^{12\bar2} + \frac{1}{2} \, r^{2}  \varphi^{1\bar1\bar2} + \frac{\sqrt{-1}}{2} \, \overline{u}  \varphi^{2\bar1\bar2} \right\rangle \oplus \mathbb C \langle \varphi^{12\bar1\bar2} \rangle ,\\
\end{eqnarray*}

We explicitly notice that, {\itshape{on primary Kodaira surfaces, an invariant Hermitian metric is never geometrically formal in the sense of Kotschick}}: indeed, $\varphi^1 \wedge \varphi^{\bar1\bar2}$ is never harmonic.

As for Chern-Ricci flow, the primary Kodaira surface has trivial canonical bundle, therefore $\mathrm{Ric}^{Ch}(\omega)=0$. Then, clearly, the Chern-Ricci flow does not evolve invariant Hermitian metrics.
\subsection*{Case 2: Secondary Kodaira surface}
Secondary Kodaira surfaces $X$ are quotients of primary Kodaira surfaces by finite groups; they have torsion non-trivial canonical bundle.

We recall the description of secondary Kodaira surfaces as quotients of solvable Lie groups \cite{hasegawa-jsg}.
There exists a coframe $(\varphi^1,\varphi^2)$ of $(1,0)$-forms with structure equations
$$ d\varphi^1 = -\frac{1}{2}\varphi^1\wedge\varphi^2+\frac{1}{2}\varphi^1\wedge\bar\varphi^2 , \qquad d\varphi^2=\frac{\sqrt{-1}}{2}\varphi^1\wedge\bar\varphi^1 .$$
The cohomologies of $X$ can be explicitly described:
\begin{eqnarray*}
H^\bullet_{dR}(X;\mathbb R) &=& \mathbb C \langle 1 \rangle \oplus \mathbb C \langle \varphi^2-\varphi^{\bar2} \rangle \oplus \mathbb C \langle \varphi^{12\bar1}-\varphi^{1\bar1\bar2} \rangle \oplus \mathbb C \langle \varphi^{12\bar1\bar2} \rangle,
\end{eqnarray*}
where we have listed the harmonic representatives with respect to the Hermitian metric with fundamental form $\omega=\sqrt{-1}\varphi^{1\bar1}+\sqrt{-1}\varphi^{2\bar2}$ instead of their classes.

As for the harmonic representatives for de Rham cohomology, the situation is very similar to the Inoue case. We list the harmonic representatives with respect to the arbitrary Hermitian metric as in \eqref{eq:metric}:
\begin{eqnarray*}
H^\bullet_{dR}(X;\mathbb R) &=& \mathbb C \langle 1 \rangle \oplus \mathbb C \langle \varphi^2-\varphi^{\bar2} \rangle \\
&& \oplus \mathbb C \left\langle -\frac{1}{2} r^2 \varphi^{12\bar1} + \frac{\sqrt{-1}}{2}  u \varphi^{12\bar2} + \frac{1}{2} r^2 \varphi^{1\bar1\bar2} + \frac{\sqrt{-1}}{2} \overline{u} \varphi^{2\bar1\bar2} \right\rangle \oplus \mathbb C \langle \varphi^{12\bar1\bar2} \rangle.
\end{eqnarray*}
We conclude that {\itshape any invariant Hermitian metric on an secondary Kodaira is geometrically formal in the sense of Kotschick}.

As for the Chern-Ricci flow, the secondary Kodaira surface has torsion canonical bundle, therefore $\mathrm{Ric}^{Ch}(\omega)=0$. Therefore, the Chern-Ricci flow does not evolve invariant Hermitian metrics.
\end{proof}

\section{Dolbeault and Bott-Chern geometric formalities}
As for Dolbeault or Bott-Chern geometric formality, the situation is clear for Hopf, Inoue and Kodaira surfaces, as we now describe. We also make computations regarding Aeppli cohomology and harmonic representatives with respect to the Aeppli Laplacian.
\begin{prop}\label{prop:db-bc}
On Hopf, Inoue, and Kodaira surfaces, the property of Dolbeault geometric formality and of Bott-Chern geometric formality is preserved along the Chern-Ricci flow starting at invariant metrics. In the same situation, the properties of having a structure of algebra or a structure of $H_{BC}$-module for harmonic-Aeppli forms are all preserved by the Chern Ricci flow.
\end{prop}

\begin{proof}
We refer to the complex structures used in Theorem \ref{thm:main} and Proposition \ref{prop:Kodaira-1}, for the computation of Dolbeault, Bott-Chern, and Aeppli cohomologies.
\subsection*{Hopf surfaces}

The Dolbeault cohomology of Hopf surfaces is explicitly described in \cite[Appendix II, Theorem 9.5]{hirzebruch}, and see \cite[Section 3.1]{angella-dloussky-tomassini} for the Bott-Chern cohomology:
\begin{eqnarray*}
H^{\bullet,\bullet}_{\overline\partial}(X) &=& \mathbb C \langle 1 \rangle \oplus \mathbb C \langle \varphi^{\bar{2}} \rangle \oplus \mathbb C \langle \varphi^{12\bar1} \rangle \oplus \mathbb C \langle \varphi^{12\bar1\bar2} \rangle , \\
H^{\bullet,\bullet}_{BC}(X) &=& \mathbb C \langle 1 \rangle \oplus \mathbb C \langle \varphi^{1\bar{1}} \rangle \oplus \mathbb C \langle \varphi^{12\bar1} \rangle \oplus \mathbb C \langle \varphi^{1\bar1\bar2} \rangle \oplus \mathbb C \langle \varphi^{12\bar1\bar2} \rangle , \\
H^{\bullet,\bullet}_{A}(X) &=& \mathbb C \langle 1 \rangle \oplus \mathbb C \langle \varphi^{2} \rangle \oplus \mathbb C \langle \varphi^{\bar2} \rangle \oplus \mathbb C \langle \varphi^{2\bar2} \rangle \oplus \mathbb C \langle \varphi^{12\bar1\bar2} \rangle ,
\end{eqnarray*}
where we have listed the harmonic representatives with respect to the Hermitian metric with fundamental form $2\omega=\sqrt{-1}\varphi^{1\bar1}+\sqrt{-1}\varphi^{2\bar2}$ instead of their classes.

We look at how the harmonic representatives of such cohomologies change with respect to the invariant Hermitian metric, and in particular when their product is still harmonic.

We summarize them as follows:
\begin{eqnarray*}
H^{\bullet,\bullet}_{\overline\partial}(X) &=& \mathbb C \langle 1 \rangle \oplus \mathbb C \langle \varphi^{\bar{2}} \rangle \oplus \mathbb C \langle -\frac{1}{2}r^2 \varphi^{1\bar1\bar2} - \frac{\sqrt{-1}}{2} \overline{u} \varphi^{2\bar1\bar2} \rangle \oplus \mathbb C \langle \varphi^{12\bar1\bar2} \rangle , \\
H^{\bullet,\bullet}_{BC}(X) &=& \mathbb C \langle 1 \rangle \oplus \mathbb C \langle \varphi^{1\bar{1}} \rangle \oplus \mathbb C \langle -\frac{1}{2} r^2 \varphi^{12\bar1} + \frac{\sqrt{-1}}{2} u \varphi^{12\bar2} \rangle \oplus \mathbb C \langle -\frac{1}{2} r^2 \varphi^{1\bar1\bar2}- \frac{\sqrt{-1}}{2} \overline{u} \varphi^{2\bar1\bar2} \rangle \oplus \mathbb C \langle \varphi^{12\bar1\bar2} \rangle , \\
H^{\bullet,\bullet}_{A}(X) &=& \mathbb C \langle 1 \rangle \oplus \mathbb C \langle \varphi^{2} \rangle \oplus \mathbb C \langle \varphi^{\bar2} \rangle \oplus \mathbb C \langle s^4\varphi^{2\bar2}+|u|^2\varphi^{1\overline{1}}-\sqrt{-1}s^2\,u\varphi^{1\overline{2}}+\sqrt{-1}s^2\overline{u}\varphi^{2\overline{1}} \rangle \oplus \mathbb C \langle \varphi^{12\bar1\bar2} \rangle ,
\end{eqnarray*}

Let us focus first on the Dolbeault cohomology. Here, the only Dolbeault-harmonic representative that changes is for the generator in $H^{1,2}_{\overline\partial}(X)$. We conclude that {\em any invariant Hermitian metric on the Hopf surface is geometrically-Dolbeault formal}.

As regards the Bott-Chern cohomology, to our aim, that is, studying harmonicity of products of Bott-Chern-harmonic forms, the only case of interest is the product $[\varphi^{1\bar1}]\smile[\varphi^{1\bar1}]$, the products with the class $[1]$ being trivial and the other products being zero because of degree reasons. Since the harmonic representatives with respect to invariant metrics are invariant, the Bott-Chern class $[\varphi^{1\overline{1}}]=\{\varphi^{1\overline{1}}+\del\delbar c \;:\; c \in \mathbb R\}$ contains only one invariant representantive, that is also harmonic with respect to any invariant Hermitian metric.
Again, we have that {\em any invariant Hermitian metric on the Hopf surface is geometrically-Bott-Chern formal}.

We consider the Aeppli cohomology. On the one side, we can consider the products between Aeppli-harmonic forms: the only possibly non-trivial products concern the classes $[\varphi^2]$ and $[\bar\varphi^2]$, $[\varphi^2]$ and $[\varphi^2\wedge\bar\varphi^2]$, $[\bar\varphi^2]$ and $[\varphi^2\wedge\bar\varphi^2]$. Since the classes $[\varphi^2]$ and $[\bar\varphi^2]$ contain only one invariant representative, we are reduced to study how the harmonic representative of the Aeppli cohomology class $[\varphi^2\wedge\bar\varphi^2]$ depends on the invariant Hermitian metric.
The arbitrary representative in the Aeppli cohomology class $[\varphi^{2\bar2}]$ is
\begin{eqnarray*}
h&:=&\varphi^{2}\wedge\bar\varphi^{2}+\partial\left( \lambda_1\bar\varphi^1+\lambda_2\bar\varphi^2 \right)+\overline\partial\left( \lambda_3\varphi^1+\lambda_4\varphi^2 \right) \\
&=& \varphi^{2\bar2}-\sqrt{-1}\left(\lambda_2+\lambda_4\right)\varphi^{1\bar1}+\sqrt{-1}\lambda_1\varphi^{2\bar1}+\sqrt{-1}\lambda_3\varphi^{1\bar2},
\end{eqnarray*}
where $\lambda_1,\lambda_2,\lambda_3,\lambda_4\in\mathbb C$.
By \eqref{eq:star-11}, we compute
\begin{eqnarray*}
V \cdot \ast_g h&=&V \cdot \left( \ast_g\,\varphi^{2\overline{2}}+\sqrt{-1}\lambda_1\ast_g\varphi^{2\overline{1}}-\sqrt{-1}\left(\lambda_2+\lambda_4\right)\ast_g\varphi^{1\overline{1}}+\sqrt{-1}\lambda_3\ast_g\varphi^{1\overline{2}} \right) \\
&=& \left(r^4-\lambda_1ur^2-\sqrt{-1}(\lambda_2+\lambda_4)|u|^2+\lambda_3 \overline{u} r^2\right)\varphi^{1\overline{1}} \\
&& + \left(-\sqrt{-1} ur^2+\sqrt{-1}\lambda_1 u^2-(\lambda_2+\lambda_4)us^2-\sqrt{-1}\lambda_3r^2s^2\right) \varphi^{1\overline{2}} \\
&&
+\left(\sqrt{-1}\overline{u}r^2-\sqrt{-1}\lambda_1r^2s^2+(\lambda_2+\lambda_4)\overline{u}s^2+\sqrt{-1}\lambda_3 \overline{u}^2\right) \varphi^{2\overline{1}} \\
&& + \left(|u|^2-\lambda_1us^2-\sqrt{-1}(\lambda_2+\lambda_4)s^4+\lambda_3\overline{u}s^2 \right)\varphi^{2\overline{2}} .
\end{eqnarray*}
By using the structure equations, we now compute
\begin{eqnarray*}
\partial(\ast_gh) &=& \sqrt{-1}\frac{-\sqrt{-1}ur^2+\sqrt{-1}\lambda_1u^2-(\lambda_2+\lambda_4)us^2-\sqrt{-1}\lambda_3r^2s^2}{r^2s^2-|u|^2}\varphi^{12\overline{2}}\\
&&-\sqrt{-1}\frac{|u|^2-\lambda_1us^2-\sqrt{-1}(\lambda_2+\lambda_4)s^4+\lambda_3\overline{u}s^2}{r^2s^2-|u|^2}\varphi^{12\overline{1}},\\
\overline\partial(\ast_gh)&=&\sqrt{-1}\frac{\sqrt{-1}\overline{u}r^2-\sqrt{-1}\lambda_1r^2s^2+(\lambda_2+\lambda_4)\overline{u}s^2+\sqrt{-1}\lambda_3\overline{u}^2}{r^2s^2-|u|^2}\varphi^{2\overline{12}}\\
&&-\sqrt{-1}\frac{|u|^2-\lambda_1us^2-\sqrt{-1}(\lambda_2+\lambda_4)s^4+\lambda_3\overline{u}s^2}{r^2s^2-|u|^2}\varphi^{1\overline{12}}.
\end{eqnarray*}
Therefore the Aeppli-harmonicity conditions $\partial\overline\partial h=\partial\ast_gh=\overline\partial\ast_gh=0$ yield
\begin{equation*}
\begin{pmatrix}
\sqrt{-1}u^2 & -us^2 & -\sqrt{-1}r^2s^2 & -us^2\\
-us^2 & -\sqrt{-1}s^4 & \overline{u}s^2 & -\sqrt{-1}s^4\\
-\sqrt{-1}r^2s^2 & \overline{u}s^2 & \sqrt{-1}\overline{u}^2 & \overline{u}s^2\\
-us^2 & -\sqrt{-1}s^4 & \overline{u}s^2 & -\sqrt{-1}s^4
\end{pmatrix}
\begin{pmatrix}
\lambda_1\\
\lambda_2\\
\lambda_3\\
\lambda_4
\end{pmatrix}=
\begin{pmatrix}
\sqrt{-1}ur^2\\
-|u|^2\\
-\sqrt{-1}\overline{u}r^2\\
-|u|^2
\end{pmatrix},
\end{equation*}
where the rank of the first matrix is $3$ thanks to the condition $r^2s^2-|u|^2>0$. By solving the system, we get
\begin{eqnarray*}
\lambda_1&=& \frac{\overline{u}}{s^2},\\
\lambda_2&=&\frac{\sqrt{-1}|u|^2}{s^4}-\lambda,\\
\lambda_3&=&-\frac{u}{s^2},\\
\lambda_4&=&\lambda,
\end{eqnarray*}
varying $\lambda\in\mathbb C$. We finally get that the harmonic representative of $[\varphi^{2\bar2}]$ with respect to $g$ is
$$ h = \frac{|u|^2}{s^4}\varphi^{1\overline{1}}-\frac{\sqrt{-1}u}{s^2}\varphi^{1\overline{2}}+\frac{\sqrt{-1}\overline{u}}{s^2}\varphi^{2\overline{1}}+\varphi^{2\overline{2}} .$$
At the end of the day, we get that: {\em Aeppli-hamornic forms have a structure of algebra if and only if the metric \eqref{eq:metric} is diagonal, namely, $u=0$}.

Finally, we consider the Aeppli cohomology as a Bott-Chern-cohomology-module. By the Stokes theorem, there is no invariant exact $4$-form other than the zero form; in particular, any invariant $(2,2)$-form is harmonic with respect to any Hermitian invariant metric. This reduces to consider only the products $[\varphi^{1\bar1}]_{BC}\smile[\varphi^2]_{A}$ and $[\varphi^{1\bar1}]_{BC}\smile[\bar\varphi^2]_{A}$. By the argument above, both $[\varphi^{1\bar1}]_{BC}$ and $[\varphi^2]_{A}$, respectively $[\bar\varphi^2]_{A}$, contain only one invariant representative that is harmonic with respect to any Hermitian metric. Therefore: {\em for any invariant Hermitian metric on the Hopf surface, Aeppli-harmonic forms have a structure of module over Bott-Chern-harmonic forms}.

As regards the Chern-Ricci flow, we already have an expression for it computed in (\ref{eq:crf-hopf}). Clearly, then, on the Hopf surface with invariant Hermitian metrics, \emph{the properties of geometric-Dolbeault formality, of geometric-Bott-Chern formality, of the Aeppli-harmonic forms having a structure of algebra, of the Aeppli-hamornic forms having a structure of module over Bott-Chern-harmonic forms, are all preserved along the Chern-Ricci flow}.

\subsection*{Inoue-Bombieri surfaces of type $S_M$}
The cohomologies of Inoue-Bombieri surfaces of type $S_M$ can be explicitly described \cite{teleman}, see \cite[Theorem 4.1]{angella-dloussky-tomassini}:
\begin{eqnarray*}
H^{\bullet,\bullet}_{\overline\partial}(X) &=& \mathbb C \langle 1 \rangle \oplus \mathbb C \langle \varphi^{\bar{2}} \rangle \oplus \mathbb C \langle \varphi^{12\bar1} \rangle \oplus \mathbb C \langle \varphi^{12\bar1\bar2} \rangle , \\
H^{\bullet,\bullet}_{BC}(X) &=& \mathbb C \langle 1 \rangle \oplus \mathbb C \langle \varphi^{2\bar{2}} \rangle \oplus \mathbb C \langle \varphi^{12\bar1} \rangle \oplus \mathbb C \langle \varphi^{1\bar1\bar2} \rangle \oplus \mathbb C \langle \varphi^{12\bar1\bar2} \rangle , \\
H^{\bullet,\bullet}_{A}(X) &=& \mathbb C \langle 1 \rangle \oplus \mathbb C \langle \varphi^{2} \rangle \oplus \mathbb C \langle \varphi^{\bar2} \rangle \oplus \mathbb C \langle \varphi^{1\bar1} \rangle \oplus \mathbb C \langle \varphi^{12\bar1\bar2} \rangle ,
\end{eqnarray*}
where we have listed the harmonic representatives with respect to the Hermitian metric with fundamental form $2\omega=\sqrt{-1}\varphi^{1\bar1}+\sqrt{-1}\varphi^{2\bar2}$ instead of their classes.

We list the harmonic representatives with respect to the arbitrary Hermitian metric as in \eqref{eq:metric}:
\begin{align}\label{eq:harm-SM}
H^{\bullet,\bullet}_{\overline\partial}(X) &= \mathbb C \langle 1 \rangle \oplus \mathbb C \langle \varphi^{\bar{2}} \rangle \oplus \mathbb C \left\langle -\frac{1}{2} r^2\varphi^{1\bar1\bar2}- \frac{\sqrt{-1}}{2} \overline{u} \varphi^{2\bar1\bar2} \right\rangle \oplus \mathbb C \langle \varphi^{12\bar1\bar2} \rangle ,\notag \\
H^{\bullet,\bullet}_{BC}(X) &=\mathbb C \langle 1 \rangle \oplus \mathbb C \langle \varphi^{2\bar{2}} \rangle \oplus \mathbb C \left\langle -\frac{1}{2} r^2\varphi^{1\bar1\bar2}- \frac{\sqrt{-1}}{2} \overline{u} \varphi^{2\bar1\bar2} \right\rangle\notag \\
& \oplus \mathbb C \left\langle -\frac{1}{2} r^2 \varphi^{12\bar1} + \frac{\sqrt{-1}}{2} u \varphi^{12\bar2} \right\rangle \oplus \mathbb C \langle \varphi^{12\bar1\bar2} \rangle , \\
H^{\bullet,\bullet}_{A}(X) &=\mathbb C \langle 1 \rangle \oplus \mathbb C \langle \varphi^{2} \rangle \oplus \mathbb C \langle \varphi^{\bar2} \rangle \notag\\ 
& \oplus \mathbb C \left\langle \varphi^{1\bar1} - \frac{\sqrt{-1} \, u}{r^{2}}  \varphi^{1\bar2} + \frac{\sqrt{-1} \, \overline{u}}{r^{2}}  \varphi^{2\bar1} + \frac{|u|^2}{r^4} \varphi^{2\bar2} \right\rangle \oplus \mathbb C \langle \varphi^{12\bar1\bar2} \rangle ,\notag
\end{align}

We conclude that: {\em any invariant Hermitian metric on an Inoue surface of type $S_M$ is geometrically-Dolbeault formal, is geometrically-Bott-Chern formal, and the Aeppli-harmonic forms have a structure of module over Bott-Chern-harmonic forms. On the other hand, Aeppli-harmonic forms have a structure of algebra if and only if the metric is diagonal}.

The Chern-Ricci flow has expression as in (\ref{eq:crf-SM}).
Clearly, we can state that on an Inoue surface of type $S_M$ with invariant Hermitian metrics, \emph{the properties of Dolbeault-geometric formality, of Bott-Chern-geometric formality, of the Aeppli-harmonic forms having a structure of algebra, of the Aeppli-hamornic forms having a structure of module over Bott-Chern-harmonic forms, are all preserved along the Chern-Ricci flow.}

\subsection*{Inoue surfaces of type $S^{\pm}$}
The cohomologies of Inoue surfaces of type $S^{\pm}$ can be explicitly described \cite{teleman}, see \cite[Theorem 4.1]{angella-dloussky-tomassini}:
\begin{eqnarray*}
H^{\bullet,\bullet}_{\overline\partial}(X) &=& \mathbb C \langle 1 \rangle \oplus \mathbb C \langle \varphi^{\bar{2}} \rangle \oplus \mathbb C \langle \varphi^{12\bar1} \rangle \oplus \mathbb C \langle \varphi^{12\bar1\bar2} \rangle , \\
H^{\bullet,\bullet}_{BC}(X) &=& \mathbb C \langle 1 \rangle \oplus \mathbb C \langle \varphi^{2\bar{2}} \rangle \oplus \mathbb C \langle \varphi^{12\bar1} \rangle \oplus \mathbb C \langle \varphi^{1\bar1\bar2} \rangle \oplus \mathbb C \langle \varphi^{12\bar1\bar2} \rangle , \\
H^{\bullet,\bullet}_{A}(X) &=& \mathbb C \langle 1 \rangle \oplus \mathbb C \langle \varphi^{2} \rangle \oplus \mathbb C \langle \varphi^{\bar2} \rangle \oplus \mathbb C \langle \varphi^{1\bar1} \rangle \oplus \mathbb C \langle \varphi^{12\bar1\bar2} \rangle ,
\end{eqnarray*}
where we have listed the harmonic representatives with respect to the Hermitian metric with fundamental form $\omega=\sqrt{-1}\varphi^{1\bar1}+\sqrt{-1}\varphi^{2\bar2}$ instead of their classes.

As for the harmonic representatives of Dolbeault, Bott-Chern and Aeppli cohomologies, the situation is exactly as (\ref{eq:harm-SM}).

We conclude that: {\em any invariant Hermitian metric on an Inoue surface of type $S^{\pm}$ is geometrically-Dolbeault formal, is geometrically-Bott-Chern formal, and the Aeppli-harmonic forms have a structure of module over Bott-Chern-harmonic forms. On the other hand, Aeppli-harmonic forms have a structure of algebra if and only if the metric is diagonal}.

The Chern-Ricci flow has expression as in (\ref{eq:crf-Spm}). Hence, we have that on an Inoue surface of type $S^\pm$ with invariant Hermitian metrics, \emph{the properties of geometric-Dolbeault formality, of geometric-Bott-Chern formality, of the Aeppli-harmonic forms having a structure of algebra, of the Aeppli-hamornic forms having a structure of module over Bott-Chern-harmonic forms, are all preserved along the Chern-Ricci flow.}

\subsection*{Primary Kodaira surfaces} 
The cohomologies of primary Kodaira surfaces can be explicitly described \cite{teleman}, see \cite[Theorem 4.1]{angella-dloussky-tomassini}:
\begin{eqnarray*}
H^{\bullet,\bullet}_{\overline\partial}(X) &=& \mathbb C \langle 1 \rangle
\oplus \mathbb C \langle \varphi^{1} \rangle
\oplus \mathbb C \langle \varphi^{\bar1}, \varphi^{\bar{2}} \rangle
\oplus \mathbb C \langle \varphi^{12} \rangle
\oplus \mathbb C \langle \varphi^{1\bar2}, \varphi^{2\bar1} \rangle
\oplus \mathbb C \langle \varphi^{\bar1\bar2} \rangle \\
&&
\oplus \mathbb C \langle \varphi^{12\bar1}, \varphi^{12\bar2} \rangle
\oplus \mathbb C \langle \varphi^{2\bar1\bar2} \rangle
 \oplus \mathbb C \langle \varphi^{12\bar1\bar2} \rangle , \\
H^{\bullet,\bullet}_{BC}(X) &=& \mathbb C \langle 1 \rangle \oplus \mathbb C \langle \varphi^{1} \rangle \oplus \mathbb C \langle \varphi^{\bar1} \rangle
\oplus \mathbb C \langle \varphi^{12} \rangle
\oplus \mathbb C \langle \varphi^{1\bar1}, 	\varphi^{1\bar2}, \varphi^{2\bar1} \rangle
\oplus \mathbb C \langle \varphi^{\bar1\bar2} \rangle \\
&&
\oplus \mathbb C \langle \varphi^{12\bar1}, \varphi^{12\bar2} \rangle
\oplus \mathbb C \langle \varphi^{1\bar1\bar2}, \varphi^{2\bar1\bar2} \rangle
 \oplus \mathbb C \langle \varphi^{12\bar1\bar2} \rangle , \\
H^{\bullet,\bullet}_{A}(X) &=& \mathbb C \langle 1 \rangle \oplus \mathbb C \langle \varphi^{1}, \varphi^2 \rangle \oplus \mathbb C \langle \varphi^{\bar1}, \varphi^{\bar2} \rangle
\oplus \mathbb C \langle \varphi^{12} \rangle
\oplus \mathbb C \langle \varphi^{1\bar2}, \varphi^{2\bar1}, \varphi^{2\bar2} \rangle
\oplus \mathbb C \langle \varphi^{\bar1\bar2} \rangle \\
&&
\oplus \mathbb C \langle \varphi^{12\bar2} \rangle
\oplus \mathbb C \langle \varphi^{2\bar1\bar2} \rangle
 \oplus \mathbb C \langle \varphi^{12\bar1\bar2} \rangle ,
\end{eqnarray*}
where we have listed the harmonic representatives with respect to the Hermitian metric with fundamental form $\omega=\sqrt{-1}\varphi^{1\bar1}+\sqrt{-1}\varphi^{2\bar2}$ instead of their classes.

We list the harmonic representatives with respect to the arbitrary Hermitian metric as in \eqref{eq:metric}:
\begin{eqnarray*}
H^{\bullet,\bullet}_{\overline\partial}(X) &=& \mathbb C \langle 1 \rangle
\oplus \mathbb C \langle \varphi^{1} \rangle
\oplus \mathbb C \langle \varphi^{\bar1}, \varphi^{\bar{2}} \rangle
\oplus \mathbb C \langle \varphi^{12} \rangle
\oplus \mathbb C \langle \varphi^{1\bar2}-\sqrt{-1} \, s  \varphi^{1\bar1}, \varphi^{2\bar1}+\sqrt{-1} \, s  \varphi^{1\bar1} \rangle
\oplus \mathbb C \langle \varphi^{\bar1\bar2} \rangle \\
&&
\oplus\, \mathbb C \,\langle -\frac{1}{2} \, r^{2}  \varphi^{12\bar1} + \frac{\sqrt{-1}}{2} \, u  \varphi^{12\bar2} ,
\frac{1}{2} \, s^{2} \varphi^{12\bar2}  + \frac{\sqrt{-1}}{2} \, \overline{u}  \varphi^{12\bar1} \rangle \\
&& 
\oplus \,\mathbb C\, \langle \frac{1}{2} \, s^{2} \varphi^{2\bar1\bar2} - \frac{\sqrt{-1}}{2} \, u  \varphi^{1\bar1\bar2} \rangle
 \oplus \mathbb C \langle \varphi^{12\bar1\bar2} \rangle , \\
H^{\bullet,\bullet}_{BC}(X) &=& \mathbb C \langle 1 \rangle \oplus\, \mathbb C \langle \varphi^{1} \rangle \oplus \mathbb C \langle \varphi^{\bar1} \rangle
\oplus \mathbb C \langle \varphi^{12} \rangle
\oplus \mathbb C \langle \varphi^{1\bar1}, \varphi^{1\bar2}, \varphi^{2\bar1} \rangle
\oplus \mathbb C \langle \varphi^{\bar1\bar2} \rangle \\
&&
\oplus\, \mathbb C\, \langle -\frac{1}{2} \, r^{2}  \varphi^{12\bar1} + \frac{\sqrt{-1}}{2} \, u  \varphi^{12\bar2} ,
\frac{\sqrt{-1}}{2} \, \overline{u}  \varphi^{12\bar1} + \frac{1}{2} \, s^{2} \varphi^{12\bar2} \rangle \\
&&
\oplus\, \mathbb C\, \langle \frac{1}{2} \, s^{2} \varphi^{2\bar1\bar2} - \frac{\sqrt{-1}}{2} \, u  \varphi^{1\bar1\bar2}, -\frac{1}{2} \, r^{2}  \varphi^{1\bar1\bar2} - \frac{\sqrt{-1}}{2} \, \overline{u}  \varphi^{2\bar1\bar2} \rangle
 \oplus \mathbb C \langle \varphi^{12\bar1\bar2} \rangle , \\
H^{\bullet,\bullet}_{A}(X) &=& \mathbb C \langle 1 \rangle \oplus \mathbb C \langle \varphi^{1}, \varphi^2 \rangle \oplus \mathbb C \langle \varphi^{\bar1}, \varphi^{\bar2} \rangle
\oplus \mathbb C \langle \varphi^{12} \rangle \oplus \mathbb C \langle \varphi^{\bar1\bar2} \rangle \\
&&
\oplus\, \mathbb C \,\langle s^2\varphi^{1\bar2}+\sqrt{-1}\,\overline{u}\, \varphi^{1\bar1},\,\,
s^{2}\varphi^{2\bar1}-\sqrt{-1}\,u\,\varphi^{1\bar1} ,\,\,
s^4\varphi^{2\bar2} -|u|^{2} \varphi^{1\bar1} \rangle
\\
&&
\oplus\, \mathbb C\, \langle \frac{1}{2} \, s^{2} \varphi^{12\bar2}  + \frac{\sqrt{-1}}{2} \, \overline{u}  \varphi^{12\bar1} \rangle
\oplus \mathbb C \langle \frac{1}{2} \, s^{2} \varphi^{2\bar1\bar2} - \frac{\sqrt{-1}}{2} \, u  \varphi^{1\bar1\bar2}  \rangle
 \oplus \mathbb C \langle \varphi^{12\bar1\bar2} \rangle ,
\end{eqnarray*}

We notice that for primary Kodaira surfaces {\itshape an invariant Hermitian metric is never geometrically-Dolbeault formal}, {\itshape e.g.} $\varphi^1 \wedge \bar\varphi^1$ is never Dolbeault-harmonic. In fact, Cattaneo and Tomassini noticed in \cite[Example 4.3]{cattaneo-tomassini} that primary Kodaira surfaces have a non-vanishing Dolbeault-Massey triple product, whence they are not Dolbeault formal in the sense of \cite{neisendorfer-taylor}.
Also, it is {\itshape never geometrically-Bott-Chern formal}, {\itshape e.g.} $\varphi^1 \wedge \varphi^{\bar1\bar2}$ is never Bott-Chern-harmonic. {\itshape The space of Aeppli-harmonic forms is never an algebra}, {\itshape e.g.} $\varphi^1\wedge\bar\varphi^1$ is never Aeppli-harmonic, {\itshape neither a module over the space of Bott-Chern harmonic forms}, {\itshape e.g.} $\varphi^1\wedge\bar\varphi^1$ is never Aeppli-harmonic.

The primary Kodaira surface has trivial canonical bundle, therefore $\mathrm{Ric}^{Ch}(\omega)=0$. Then the Chern-Ricci flow does not evolve invariant Hermitian metrics.

Then clearly on a primary Kodaira surface with invariant Hermitian metrics, \emph{the properties of geometric-Dolbeault formality, of geometric-Bott-Chern formality, of the Aeppli-harmonic forms having a structure of algebra, of the Aeppli-hamornic forms having a structure of module over Bott-Chern-harmonic forms, are all preserved along the Chern-Ricci flow.}

\subsection*{Secondary Kodaira surfaces}

The cohomologies of secondary Kodaira surfaces can be explicitly described \cite{teleman}, see \cite[Theorem 4.1]{angella-dloussky-tomassini}:
\begin{eqnarray*}
H^{\bullet,\bullet}_{\overline\partial}(X) &=& \mathbb C \langle 1 \rangle \oplus \mathbb C \langle \varphi^{\bar{2}} \rangle \oplus \mathbb C \langle \varphi^{12\bar1} \rangle \oplus \mathbb C \langle \varphi^{12\bar1\bar2} \rangle , \\
H^{\bullet,\bullet}_{BC}(X) &=& \mathbb C \langle 1 \rangle \oplus \mathbb C \langle \varphi^{1\bar{1}} \rangle \oplus \mathbb C \langle \varphi^{12\bar1} \rangle \oplus \mathbb C \langle \varphi^{1\bar1\bar2} \rangle \oplus \mathbb C \langle \varphi^{12\bar1\bar2} \rangle , \\
H^{\bullet,\bullet}_{A}(X) &=& \mathbb C \langle 1 \rangle \oplus \mathbb C \langle \varphi^{2} \rangle \oplus \mathbb C \langle \varphi^{\bar2} \rangle \oplus \mathbb C \langle \varphi^{2\bar2} \rangle \oplus \mathbb C \langle \varphi^{12\bar1\bar2} \rangle ,
\end{eqnarray*}
where we have listed the harmonic representatives with respect to the Hermitian metric with fundamental form $\omega=\sqrt{-1}\varphi^{1\bar1}+\sqrt{-1}\varphi^{2\bar2}$ instead of their classes.

As for the harmonic representatives for Dolbeault, Bott-Chern and Aeppli cohomologies, the situation is very similar to the Inoue case, only the computations for the class $[\varphi^{2\bar2}]\in H^{1,1}_{A}(X)$ being slightly different.

We list the harmonic representatives with respect to the arbitrary Hermitian metric as in \eqref{eq:metric}:
\begin{eqnarray*}
H^{\bullet,\bullet}_{\overline\partial}(X) &=& \mathbb C \langle 1 \rangle \oplus \mathbb C \langle \varphi^{\bar{2}} \rangle \oplus \mathbb C \left\langle -\frac{1}{2} r^2\varphi^{1\bar1\bar2}- \frac{\sqrt{-1}}{2} \overline{u} \varphi^{2\bar1\bar2} \right\rangle \oplus \mathbb C \langle \varphi^{12\bar1\bar2} \rangle  \\
H^{\bullet,\bullet}_{BC}(X) &=& \mathbb C \langle 1 \rangle \oplus \mathbb C \langle \varphi^{1\bar{1}} \rangle \oplus \mathbb C \left\langle -\frac{1}{2} r^2\varphi^{1\bar1\bar2}- \frac{\sqrt{-1}}{2} \overline{u} \varphi^{2\bar1\bar2} \right\rangle \\
&& \oplus \mathbb C \left\langle -\frac{1}{2} r^2 \varphi^{12\bar1} + \frac{\sqrt{-1}}{2} u \varphi^{12\bar2} \right\rangle \oplus \mathbb C \langle \varphi^{12\bar1\bar2} \rangle , \\
H^{\bullet,\bullet}_{A}(X) &=& \mathbb C \langle 1 \rangle \oplus \mathbb C \langle \varphi^{2} \rangle \oplus \mathbb C \langle \varphi^{\bar2} \rangle \\
&& \oplus \mathbb C \left\langle
|u|^2 \varphi^{1\bar1} - \sqrt{-1} \,s^2 u \varphi^{1\bar2} + \sqrt{-1} \,s^2 \overline{u}  \varphi^{2\bar1} + s^4 \varphi^{2\bar2}
\right\rangle \oplus \mathbb C \langle \varphi^{12\bar1\bar2} \rangle.
\end{eqnarray*}

We conclude that: {\em any invariant Hermitian metric on a secondary Kodaira surface is geometrically-Dolbeault formal, is geometrically-Bott-Chern formal, and the Aeppli-harmonic forms have a structure of module over Bott-Chern-harmonic forms. On the other hand, Aeppli-harmonic forms have a structure of algebra if and only if the metric is diagonal}.

The secondary Kodaira surface has torsion canonical bundle, therefore $\mathrm{Ric}^{Ch}(\omega)=0$. Then the Chern-Ricci flow does not evolve invariant Hermitian metrics.

Then clearly on a secondary Kodaira surface with invariant Hermitian metrics, \emph{the properties of geometric-Dolbeault formality, of geometric-Bott-Chern formality, of the Aeppli-harmonic forms having a structure of algebra, of the Aeppli-hamornic forms having a structure of module over Bott-Chern-harmonic forms, are all preserved along the Chern-Ricci flow.}
\end{proof}

We summarize the results in the last two Sections in Table \ref{table:summary}.

\begin{center}
\begin{table}[h!]
\begin{footnotesize}
  \begin{tabular}{c|ccccc|}
    \multirow{2}{*}{\textbf{surface}} & \textbf{Kotschick} & \textbf{Dolbeault} & \textbf{Bott-Chern} & \textbf{Aeppli harm. f.} & \textbf{Aeppli harm. f.} \\
    & \textbf{geom. form.} & \textbf{geom. form.} & \textbf{geom. form.} & \textbf{as algebra} & \textbf{as BC-module} \\
    \hline
    \hline
    \textbf{class $\mathrm{VII}^{b_2>0}$} & never & ? & ? & ? & ? \\     		\hline
	\textbf{Hopf} & always & always & always & diagonal & always \\
	{\scriptsize (invariant metrics)} &&&&& \\
	\hline
	\textbf{Inoue-Bombieri $S_M$} & always & always & always & diagonal & always \\
	{\scriptsize (invariant metrics)} &&&&& \\
	\hline
	\textbf{Inoue $S^\pm$} & always & always & always & diagonal & always \\
	{\scriptsize (invariant metrics)} &&&&& \\
	\hline
	\textbf{primary Kodaira} & never & never & never & never & never \\
	\hline
	\textbf{secondary Kodaira} & always & always & always & diagonal & always \\ 
	{\scriptsize (invariant metrics)} &&&&& \\
	\hline
  \end{tabular}
\caption{\label{table:summary}
Summary of Theorem \ref{thm:main} and Propositions \ref{prop:Kodaira-1} and \ref{prop:db-bc} concerning geometric formalities (for Kotschick, Dolbeault, Bott-Chern) and the structure of Aeppli-harmonic forms with respect to Hermitian metrics, respectively invariant Hermitian metrics on Hopf, Inoue, Kodaira surfaces.}
\end{footnotesize}
\end{table}
\end{center}

In view of further study, we notice that:
\begin{itemize}
\item in any mentioned cases, the Chern-Ricci flow starting at an invariant metric clearly preserves each one of the above properties, since it preserves diagonal metrics. (Compare also \cite[Proposition 3]{kotschick-terzic-pacific}, showing that, for certain $G$-homogeneous spaces, every $G$-invariant metrics is geometrically formal.) We ask whether this behaviour is more general, or whether there exists a counterexample for which the Chern-Ricci flow does {\em not} preserve the geometric formality in the sense of Kotschick, or any other of the geometric Hermitian formalities discussed above.
We notice that the above invariant metrics are Gauduchon, that is pluriclosed (also known as SKT) being defined on four-dimensional manifolds. Therefore, as the Referee kindly suggested to us, it may be interesting to further investigate the $6$-dimensional nilmanifolds admitting invariant SKT metrics as classified in \cite{fino-parton-salamon}.
\item Clearly, holomorphically-parallelizable manifolds do not provide such counterexamples when restricting to invariant metrics, since they have holomorphically-trivial canonical bundle, whence invariant Hermitian metrics are Chern-Ricci-flat.
Our attempts on four-dimensional Lie groups (possibly not admitting compact quotients), as in \cite{ovando} and references therein, or small deformations of the Iwasawa manifold \cite{nakamura, angella-1} still have not provided further examples.
\item The same question may be addressed for other geometric flows other than the Chern-Ricci flow, for example the Hermitian curvature flows in  \cite{streets-tian} or in particular the one studied in \cite{ustinovskiy}.
\item It could be interesting to further investigate Massey triple products and Dolbeault Massey products, see \cite{tomassini-torelli, cattaneo-tomassini}, or other Massey products, in particular on class VII surfaces with $b_2>0$ and on primary Kodaira surfaces.
\end{itemize}

\end{document}